\documentclass[11pt]{article}
\usepackage{amsfonts} 
\usepackage{amsmath}
\usepackage{amsthm}
\usepackage{epsfig}
\usepackage{amssymb}
\usepackage{psfrag}

\newcommand{\wt}{\tilde{w}}
\newcommand{\wl}{w^{(j)}}

\newcommand{\be}{\begin{eqnarray}}
\newcommand{\ee}{\end{eqnarray}}

\newcommand{\om}{\Omega}
\newcommand{\Div}{{\rm div}\,}

\newcommand{\dist}{{\rm dist}\,}

\newcommand{\1}{{\bf 1}}
\newcommand{\tr}{{\rm tr}\,}
\newcommand{\cof}{{\rm cof}\,}
\newcommand{\adj}{{\rm adj}\,}

\newcommand{\eps}{\epsilon}

\newcommand{\sch}{\mathcal{H}}

\newcommand{\scl}{\mathcal{L}}

\DeclareMathOperator*{\essinf}{ess\,inf}
\newcommand{\veps}{\varepsilon}
\newcommand{\id}{\mathbf{i}}
\newcommand{\wctwo}{\overset{\circ}{w_{2}}}
\newcommand{\wcone}{\overset{\circ}{w_{1}}}
\newcommand{\wid}{\overset{\circ}{\textbf{i}}}
\newcommand{\wc}{\overset{\circ}{w}}

\def\XXint#1#2#3{{\setbox0=\hbox{$#1{#2#3}{\int}$}
     \vcenter{\hbox{$#2#3$}}\kern-.5\wd0}}

\newtheorem{example}{Example}[section]
\newtheorem{thm}{Theorem}[section]
\newtheorem{prop}{Proposition}[section]
\newtheorem{lem}{Lemma}[section]

\newtheorem{defn}{Definition}[section]
\newtheorem{rem}[prop]{Remark}
\numberwithin{equation}{section}

\title{A remark on a stability criterion for the radial cavitating map in nonlinear elasticity}
\author{J. Bevan }

\begin{document}
\maketitle
\begin{abstract}We study the integral functional $I(w):=\int_{B}\left| \adj \nabla w \left(\frac{w}{|w|^{3}}\right)\right|^{q}\,dx$ on suitable maps $w:B \subset \mathbb{R}^{3} \to \mathbb{R}^{3}$ and where $2q \in (2,3)$.   The inequality $I(w) \geq I(\mathbf{i})$, which we establish on a subclass of the admissible maps, was first proposed in \cite{SS08} as one of two possible necessary conditions for the stability, i.e. local minimality, of the radial cavitating map in nonlinear elasticity.  Here, $\mathbf{i}$ is the identity map.  Admissible maps $w$ either do not vanish (and in this case possess a single discontinuity $x_{0}$ in B which produces a cavity about the origin), or vanish at exactly one point $x_{0}$ in B, in which case $w$ is a diffeomorphism in a neighbourhood of $x_{0}$.   We show that $I(\cdot)$ behaves like a polyconvex functional and associate with it another functional, $K(\cdot)$, satisfying $I(w) \geq I(\mathbf{i}) + q(K(w)-K(\mathbf{i}))$.   We give conditions under which $K(w) = K(\mathbf{i})$, and from these infer $I(w) \geq I(\mathbf{i})$.   It is also shown that (i) $K$ is strictly decreasing along paths of admissible functions that move $x_{0}$ \emph{away} from the origin and (ii) $K(w)$ exhibits some quite pathological behaviour when $w$ is sufficiently close to $\mathbf{i}$. \end{abstract}

\section{Introduction}
In this paper we prove that, when $q \in \left(1,\frac{3}{2}\right)$, the inequality 
\begin{equation}\label{stabnec}
\int_{B}\left| \adj \nabla w \left(\frac{w}{|w|^{3}}\right)\right|^{q}\,dx \geq \int_{B}\frac{1}{|x|^{2q}}\,dx \end{equation}
holds for all $w$ in a certain subset of the Sobolev space $W^{1,2q}(B;\mathbb{R}^{3})$, details of which are given below.    This inequality first appears in \cite[(1.4), p 203]{SS08} as a necessary condition for the stability, i.e. local minimality, of the radial cavitating map as measured by an energy of the form
\[ I(u) = \int_{B} W(\nabla u)\, dx.\]
The authors of \cite{SS08} leave the question of whether or not \eqref{stabnec} holds open.

The \emph{polyconvex} stored-energy functions $W$ defined on $3 \times 3$ matrices considered in \cite{SS08} are of the form
\[W(F)=\Psi(|F|,|\adj F|) + h(\det F),\]
where $\Psi$ is required to satisfy certain scaling and structural assumptions and $h$ should, among other things, be convex.   Although the class of such polyconvex $W$ is wide, the origin of \eqref{stabnec} can be understood by focussing on the following particular example.  Let
\begin{equation}\label{typicalw}W(F) := \left\{\begin{array}{l l}|F|^{p} + |\adj F|^{q} + h(\det F) & \textrm{if} \ \det F > 0 \\
+\infty & \textrm{otherwise}.\end{array}\right.\end{equation}
Here, the exponents should satisfy $2 \leq p <3$ and $1 \leq q < \frac{3}{2}$.

There are two cases to distinguish, based on the relative growth rates of the terms $|F|^{p}$ and $|\adj F|^{q}$.   When the term $|F|^{p}$ dominates, i.e. when $p > 2q$, the stability of the radial cavitating map is determined by another inequality---\cite[(4.8) in Corollary 4.4]{SS08}---which, owing to the efforts of a considerable number of researchers, is known to hold in cases applying to nonlinear elasticity:  see \cite[Remark 4.5]{SS08} for details.  When the $|\adj F|^{q}$ term dominates, i.e. when $p < 2q$, stability is determined by \eqref{stabnec}.   Therefore in the rest of the paper we will assume $2 < 2q < 3$.



The radial cavitating map was discovered by Ball in \cite{Ba81}.   He considered an elastic material initially occupying the ball $B$ and subject to a uniform linear displacement $x \mapsto \lambda x$ for $x \in \partial B$.   In \cite{Ba81} it is shown that, for sufficiently large $\lambda$, the material responds by forming a radially symmetric cavity at the centre of the ball.   In other words, the material deforms according to a displacement of the form 
\[u_{\textrm{radial}}(x)= r(|x|)\frac{x}{|x|} \ \ \ \ x \in B,\]
where $r(0) > 0$.  Moreover, the energy associated with $u_{\textrm{radial}}$ is strictly smaller than that of the uniform displacement 
\[u_{\lambda}(x)=\lambda x \ \ \ \ x \in B,\] 
which signifies, among other things, a loss of quasiconvexity in a pure displacement problem.   It is clear, then, that $u_{\textrm{radial}}$ is not a global minimizer among maps which agree with $u_{\lambda}$ on the boundary.  However, it is still possible that $u_{\textrm{radial}}$ is a local minimizer in some appropriate norm.  Thus the relevance of the necessary conditions presented in \cite{SS08}, one of which we study here.  
For an overview of cavitation, see \cite{Ba01},  or for more recent progress consult \cite{ST06,SN10,MCH2}.  For details of cavitation in the context of functionals with a surface energy term, the papers \cite{MuSp95,MCH1} are highly recommended.


The topological constraints on the admissible maps considered in this paper are, initially, exactly those imposed via statements (4.5) and (4.6) in \cite[Theorem 4.2]{SS08}.   Roughly speaking, and for now ignoring the invertibility constraints, they require either that deformations are bounded away from the origin or, if they do vanish, then they do so once and are a diffeomorphism in a neighbourhood of the zero.   See Definition \ref{defadmiss} for the details.   Our results apply to strengthened versions of these conditions:  see the statement of Lemma \ref{salix4} and Definition \ref{d:condtwoprime} for details.   

To study the functional $I$ we bound it below by another functional $K$, which is defined in \eqref{K}, and study $K$ instead.  Specifically, we show that
\[I(w) \geq I(\mathbf{i}) + q (K(w)-K(\mathbf{i})),\]
where $\mathbf{i}$ is the identity map and $q$ is the exponent appearing in \eqref{stabnec}.     The main feature of $K$ is that if it is minimized at the identity then \eqref{stabnec}---the central inequality of this paper---holds at $w$.   We find in Section \ref{S:nonvanish} that $K$ is constant on sufficiently regular paths of functions which satisfy a version of the topological conditions mentioned above.  This can be exploited to prove Theorem \ref{t:main1}, which says that inequality \eqref{stabnec} is true for any $w$ which can be connected by such a path to the identity map and whose (single) discontinuity or zero sits at the origin of the domain, in common with that of the identity.   Thus the location of the discontinuity or zero of $w$ is of central importance.   In this context, see the interesting results of \cite{SS99,SS00}. 

 It is possible to apply variations which `move' the singularity of a given $w$:  this is the emphasis of Section \ref{S:vanish}.   The two most significant behaviours of 
$K$ in this regard, and applying to maps which vanish at one point in the domain, are that: (i) $K(w)$ decreases if the zero of $w$ moves \emph{away} from the origin (see Proposition \ref{p:furthersmaller}), and (ii) there are maps $w$ arbitrarily close to the identity map such that the derivative of $K(w)$ along suitably chosen variations\footnote{These will turn out to be so-called inner variations.} about $w$ can be made arbitrarily large.   Point (i) suggests that the identity may not be a minimizer of $K$: see the latter part of Section \ref{S:vanish} for the details and discussion. Point (ii), which is proved in Proposition \ref{p:laster}, suggests a rather pathological behaviour near the identity, but leaves open the question of whether the identity map minimizes $K$ among maps whose zero or discontinuity is not located at the origin.

\section{Notation}

We denote the $3 \times 3$ real matrices by $\mathbb{R}^{3 \times 3}$ and the identity matrix by $\1$.  Throughout, $B$ is the unit ball in $\mathbb{R}^{3}$ and $B(a,R)$ represents the open ball in $\mathbb{R}^{3}$ centred at $a$ with radius $R$.   Other standard notation includes $||\cdot ||_{k,p;\om}$ for the norm on the Sobolev space $W^{k,p}(\om)$, and both $||\cdot||_{p;\om}=||\cdot||_{L^{p}(\om)}$ denote the norm on $L^{p}(\om)$.    Here, $\om$ is a domain in $\mathbb{R}^{3}$.  We will denote the unit vectors in $\mathbb{R}^{3}$ by $\mathbb{S}^{2}$.

The tensor product of two vectors $a \in \mathbb{R}^{3}$ and $b \in \mathbb{R}^{3}$ is written $a \otimes b$; it is the $3 \times 3$ matrix whose $(i,j)$ entry is $a_{i}b_{j}$.   The inner product of two matrices $X,Y \in \mathbb{R}^{3 \times 3}$ is $X \cdot Y = \tr(X^{T}Y)$.   
This obviously holds for vectors too.  Accordingly, we make no distinction between the norm of a matrix and that of a vector:  both are defined by $|\alpha|^{2}=\alpha \cdot \alpha$.     For any $3 \times 3$ matrix we write $\adj F= (\cof F)^{T}$, and $\det F$ denotes the determinant as usual.  Other notation relating to matrix algebra will be introduced when it is needed, most notably in Definition \ref{defo1}.  In calculations involving matrices we sum over repeated indices, unless stated otherwise.

The identity function will be written $\mathbf{i}$, and for any nonzero vector $x \in \mathbb{R}^{3}$ we define $\bar{x}=x/|x|$.    The same notation extends to maps $w:B \to \mathbb{R}^{3}$, so that we shall write $\bar{w}$ for the $\mathbb{S}^{2}-$valued function $\frac{w(x)}{|w(x)|}$ whenever $w(x) \neq 0$.  Integrals throughout this paper are nearly always with respect to three dimensional Lebesgue measure: accordingly,  $dx$ will be shorthand for $d \scl^{3}(x)$.   In other cases, the standard notation $\sch^{2}$ indicates two-dimensional Hausdorff measure.  Finally, constants $C$ appearing in inequalities may change from line to line.

\section{Stability when $w$ is bounded away from zero}\label{S:nonvanish}

We now turn to the description of the restricted class of maps in which inequality \eqref{stabnec} will be studied.    Firstly, a map $w$ in $W^{1,2q}(B;\mathbb{R}^{3})$ will be called a deformation if it satisfies $\det \nabla w > 0$ a.e. and is one-to-one almost everywhere.

We then define \emph{admissible deformations} as follows: 

\begin{defn}\label{defadmiss} Let $2q \in [1,3)$ and let $w \in W^{1,2q}(B;\mathbb{R}^{3})$ be a deformation satisfying $w=\bf{i}$ on $\partial B$.  In addition, suppose that there is a Lebesgue null set $N=N_{w} \subset B$ and an open set $U=U_{w} \subset B$ such that either
\begin{itemize}\item[(I)] $0 \in U$, \ \ $w(B \setminus N) \cap U = \emptyset$, or
 \item[(II)] $0 \in w(U)$, \ \ $w\arrowvert_{U}$ is a diffeomorphism, \ \ $w(B \setminus (U \cup N)) \cap w(U) = \emptyset$.
\end{itemize}
\end{defn}

Note that (I) and (II) are statements (4.5) and (4.6) in \cite[Theorem 4.2]{SS08}, and that the identity map $w=\bf{i}$ is admissible.

Let the functional $I(w)$ be defined by 
\begin{equation}\label{defI}
I(w) = \int_{B}\left| \adj \nabla w \left(\frac{w}{|w|^{3}}\right)\right|^{q}\,dx
\end{equation}
on admissible $w$ as described above.  Our approach to the problem of minimizing $I$ among admissible deformations is to view its integrand
\[ \Phi(w,\adj \nabla w) := \left|\adj \nabla w \left(\frac{w}{|w|^{3}}\right)\right|^{q}\]
as a convex function of the null Lagrangian 
\begin{equation}\label{G}G(w):=\adj \nabla w \left(\frac{w}{|w|^{3}}\right).\end{equation}
Here, the convexity is that of the function $f: \mathbb{R}^{3} \to \mathbb{R}$ defined by $f(k)=|k|^{q}$.  Thus 
\[\Phi(w,\adj \nabla w)=f(G(w))\] 
behaves rather like a polyconvex function:  see, e.g., \cite{Ba77,Ba01,Da08} for more on polyconvex functions in the classical sense.

The term null Lagrangian refers to a quantity whose Euler-Lagrange equation holds as an identity on smooth maps, i.e if
\begin{equation}\label{NL0}\Div D_{F}L(x,w,\nabla w) = D_{w}L(x,w,\nabla w)\end{equation}
holds for all smooth $w$.  In the $3 \times 3$ case, well known examples include $L_{1}(F)=F$, $L_{2}(F) = \adj F$, and $L_{3}(F) =\det F$.  See \cite{BCO,Da08} for much more on null Lagrangians.
In fact, \eqref{NL0} holds when $L$ is replaced by $G(w)$ provided  
$w$ is suitably smooth on $B \setminus \{0\}$ and $|w(x)| \geq c > 0$ for all nonzero $x$ in the ball $B$:  this can be seen by a calculation using the machinery of Section \ref{auxil}.  It cannot, however, be directly exploited because $\Phi(v,F)$ is not convex in $(v,F)$.  Instead, we use the convexity of $f$, as indicated above: see Prop \ref{salix1} for details.

The functional $I$ defined in \eqref{defI} exhibits some interesting behaviour, as the following examples show.  

\begin{example}\label{e:1}\emph{Consider maps of the form 
\[u(x)=r(|x|)\bar{x}\ \ \ \ x \in B.\]
Assume that $u$ lies in $W^{1,2q}(B,\mathbb{R}^{3})$ for $2q \in (2,3)$, and let $r(1)=1$, so that $u=\mathbf{i}$ on $\partial B$.
With $R=|x|$ and $\bar{x}=x/|x|$, 
\[\nabla u(x) = \left(r'(R)-\frac{r(R)}{R}\right) \bar{x} \otimes \bar{x} +\frac{r(R)}{R}\1.\]
Hence, 
\[ \adj \nabla u = \frac{rr'}{R}\1 + \left(\frac{r^{2}}{R^{2}}-\frac{rr'}{R}\right)\bar{x} \otimes \bar{x},\]
from which it follows that
\[ \adj \nabla u  \left(\frac{u}{|u|^{3}}\right) = \frac{\bar{x}}{R^2}.\]
In particular,
\begin{equation*}\begin{split}I(u) & = \int_{B} \left|\frac{\bar{x}}{R^2}\right|^{q}\,dx \\ 
& = 4 \pi \int_{B} R^{2-2q}\,dR \\
& = I(\bf{i}).\end{split}\end{equation*}}
\end{example}

The next example corresponds to a specific choice of $r$ in Example \ref{e:1}.  

\begin{example}\label{e:2}\emph{Consider the following map, which opens a cavity of radius $\lambda$ at the origin.   Fix $0 < \lambda < 1$ and define 
\begin{displaymath}w(x)=\left\{\begin{array}{l l}\lambda \bar{x} \ \  & \textrm{if} \ \ 0 < R < \lambda \\
 \mathbf{i} \ \ &  \textrm{if} \ \  \lambda \leq R \leq 1. \end{array}\right.\end{displaymath}
Applying Example \ref{e:1} with 
\begin{displaymath}r(R)=\left\{\begin{array}{l l}\lambda \ \  & \textrm{if} \ \ 0 < R < \lambda \\
 R  \ \ &  \textrm{if} \ \  \lambda \leq R \leq 1, \end{array}\right.\end{displaymath}
it follows that $I(w)=I(\mathbf{i})$.  Thus $I(\cdot)$ does not appear to distinguish between competitors $w$ that do not vanish and the identity  $\mathbf{i}$, which does.  Cf. conditions (I) and (II) in Definition \ref{defadmiss}.}\end{example}

Examples \ref{e:1} and \ref{e:2} show in particular that if $\mathbf{i}$ is indeed the minimizer of the functional $I(\cdot)$ then it is not unique in this regard.

\subsection{An auxiliary inequality}\label{auxil}

We now exploit the `polyconvexity' of the functional $I$.  

\begin{prop}\label{salix1} Let $w$ be admissible in the sense of Definition \ref{defadmiss}, and let the functional $I$ be given by 
\[I(w)=  \int_{B} f(G(w))  \,dx, \]
 with $f(k)=|k|^{q}$ and 
 \[G(w) =\adj \nabla w \left(\frac{w}{|w|^{3}}\right)\]
as in \eqref{G}.   Let $q$ satisfy $2< 2q < 3$.
Then
\begin{equation}\label{LB} I(w) \geq I(\mathbf{i}) + q \int_{B} \frac{x}{|x|^{2q -1}} \cdot G(w) \,dx - q \int_{B}\frac{1}{|x|^{2q}}\,dx.\end{equation} 
\end{prop}

\begin{proof} First note that $G(\mathbf{i})=\frac{x}{|x|^{3}}$ satisfies $|G(\mathbf{i})| \geq 1$ for all $x$ in $B$.  Since $q \geq 1$, $f$ is convex, and since $G(\mathbf{i})$ is bounded away from zero, it follows that
\begin{equation}\label{conv} f(G(w)) \geq f(G(\mathbf{i})) + Df(G(\mathbf{i}))\cdot(G(w)-G(\mathbf{i})) \end{equation}
holds pointwise.   

Now 
\[Df(G(\mathbf{i}))=  \frac{qx}{|x|^{2q-1}},\]
and so the function 
\[Df(G(\mathbf{i})) \cdot G(\mathbf{i}) = \frac{q}{|x|^{2q}}\]
is clearly integrable.  

In order to check the integrability of the term $Df(G(\mathbf{i})) \cdot G(w)$ we distinguish two cases.

Firstly, suppose (I) holds.  Then we may assume that $w/|w|^{3}$ is essentially bounded on $B$, and so by H\"{o}lder's inequality
\[ ||Df(G(\mathbf{i})) \cdot G(w)||_{1}  \leq  C ||x/|x|^{2q-1}||_{q'}|||\nabla w|^{2}||_{q},\]
where $q'$ is the H\"{o}lder conjugate to $q>1$ and $C$ is a constant depending only on $||w^{-2}||_{\infty}$ and $q$.   It is easily verified that 
$||x/|x|^{2q-1}||_{q'}$ is finite exactly when $2q<3$.   In particular, we note that both endpoints of the interval $2 \leq 2q \leq 3$ are excluded by this procedure.

Secondly, suppose that (II) holds.   Thus we may assume that there is an open set $U \subset B$ such that $w(U)$ is open and contains $0$, and that, in addition, $(w\arrowvert_{B\setminus U})^{-1}$ is essentially bounded.   The integral
\begin{equation}\label{BwithoutU}\int_{B\setminus U} |Df(G(\mathbf{i})) \cdot G(w)|\,dx \end{equation}
can therefore be estimated using the method that was applied to case (I).   Now $w$ has a zero in $U$.   Suppose first that $w(x_{0})=0$ and $x_{0} \neq 0$.   Without loss, we can also suppose $\dist(U,0)>0$, so that $\frac{x}{|x|^{2q-1}}$ is bounded on $U$.  Moreover, since $\nabla w$ inverts on $U$, there are constants $c<C$ depending on $U$ and $w$ such that
\begin{equation}\label{uplow} c|x-x_{0}| \leq |w(x)| \leq C|x-x_{0}| \ \ \textrm{if} \ x \in  U.\end{equation}
Since $\nabla w$ is by hypothesis continuous on $U$, it is now easy, in view of \eqref{uplow}, to estimate the integral in \eqref{BwithoutU}.   The only other case is $x_{0}=0$.  But then the integrability of \eqref{BwithoutU} is determined by that of the product $Df(G(\mathbf{i})) \cdot \frac{w}{|w|^{3}}$, which is bounded above by $1/|x|^{2q}$.  This is integrable because $2q <3$.  Therefore in cases (I) and (II) both sides of inequality \eqref{conv} are integrable.   Integrating \eqref{conv} gives \eqref{LB}.   \end{proof}

Inequality \eqref{LB} tells us that
\begin{equation}\label{morus}
I(w) \geq I(\id) +  q(K(w)-K(\id)), 
\end{equation}
where the functional $K$ is defined by
\begin{equation}\label{K} K(w):=\int_{B} \frac{x}{|x|^{2q -1}} \cdot G(w) \,dx. \end{equation}
It follows immediately from \eqref{morus} that if $K(w)=K(\id)$ then  $I(w) \geq I(i)$, which is the desired inequality.    Therefore our strategy is to show that $K(\gamma(\cdot;t))$ is constant on appropriately chosen paths of maps $(\gamma(\cdot;t))_{t \in  [0,1]}$, where  each $\gamma(\cdot;t): B \to B$ satisfies $\gamma(x;t)=x$ if $|x| =1$, $0 \leq t \leq 1$, and where $\gamma(\cdot;0)=w$, $\gamma(\cdot;1)=\id$.    While it is clearly important that the map $w$ is an admissible deformation in the sense of Definition \ref{defadmiss},  the same constraint need not apply to maps $\gamma(\cdot;t)$ `along the path'.   Intervening $\gamma(\cdot;t)$ need not be deformations, for example.   This leeway is exploited in Theorem \ref{t:main1} and in Section \ref{pctoid}.  

We begin by calculating the first variation of $K$, for which some notation is needed:  we introduce this below, together with two useful algebraic lemmas.

\begin{defn}\label{defo1}
Let $\xi$ and $\eta$ be two matrices in $\mathbb{R}^{3 \times 3}$.  We define the $3 \times 3$ matrix  $\langle \xi,\eta \rangle$ by
\[ \langle \xi,\eta \rangle = \partial_{t}\arrowvert_{t=0}\adj (\xi + t\eta).\]
In terms of the alternating symbol on three elements,
\[ \langle \xi,\eta \rangle_{ij}= \veps^{jab}\veps^{icd}\xi_{ac}\eta_{bd} \ \ \ \ \  1 \leq i,j \leq 3.\]
Equivalently, 
\[\adj(\xi + \eta)=\adj \xi + \langle \xi, \eta \rangle + \adj \eta.\]
\end{defn}

Note that $\langle \xi,\eta \rangle = \langle \eta,\xi \rangle$ for all $\xi, \eta$ and that $\langle \cdot\,,\cdot \rangle$ is a bilinear form.

\begin{lem}
For all $\xi \in \mathbb{R}^{3 \times 3}$ and all vectors $u,v \in \mathbb{R}^{3}$, 
\begin{equation}
\label{damson}\langle \xi, u \otimes v \rangle = \adj (\xi + u \otimes v) - \adj \xi.  
\end{equation}
In particular, 
\begin{equation}\label{jamson} \adj(\1 + u \otimes v) = (1+u \cdot v) \1 -u \otimes v
\end{equation}
 \end{lem}                                                                                    
 \begin{proof}
 Identity \eqref{damson} follows by noting that 
 \[ \adj (\xi + u \otimes v) = \adj \xi + \langle \xi, u \otimes v \rangle + \adj \, u \otimes v\] 
 and using the fact that $\adj \,u \otimes v = 0$.
To see \eqref{jamson}, apply \eqref{damson} with $\xi = \1$ to obtain
\[\adj(\1+u \otimes v) - \1 =  \langle \1,u \otimes v \rangle.\] 
The right-hand side is easily calculated, using Definition \ref{defo1}, to be $u\cdot v \1 - u \otimes v$. 
 \end{proof}

\begin{lem}\label{l:identity}
For all $F \in \mathbb{R}^{3 \times 3}$ and all $v \in \mathbb{S}^{2}$, 
\begin{equation}\label{ramson} \adj F = \adj F v \otimes v + \langle F, v \otimes F^{T}v\rangle.
 \end{equation}\end{lem}
\begin{proof}
Let $F$ be invertible.  By \eqref{damson} and \eqref{jamson},
\begin{eqnarray*} \langle F, v \otimes F^{T}v \rangle & = & \adj (F + v \otimes F^{T}v) - \adj F \\
& = & \adj(F.[\1 + F^{-1}v \otimes F^{T}v]) - \adj F \\
& = &  \adj(\1 + F^{-1}v \otimes F^{T}v)\adj F - \adj F \\
& = & ((1+F^{-1}v \cdot F^{T}v)\1 -F^{-1}v \otimes F^{T}v)\adj F - \adj F \\
& = &  (F^{-1}v \cdot F^{T}v -F^{-1}v \otimes F^{T}v)\adj F.
\end{eqnarray*}
Let $\{v,\tau_{1},\tau_{2}\}$ be an orthonormal basis for $\mathbb{R}^{3}$ and write
\[ \1 = v \otimes v + \tau_{1} \otimes \tau_{1} + \tau_{2} \otimes \tau_{2}.\]
The identity \eqref{ramson} is then equivalent to 
\[ \adj F (\tau_{1} \otimes \tau_{1} + \tau_{2} \otimes \tau_{2}) = \langle F, v \otimes F^{T}v \rangle,\]
or
\begin{equation}\label{jettison} \adj F \tau = \langle F, v \otimes F^{T}v \rangle \tau\end{equation}
for any $\tau \in \mathbb{R}^{2}$ satisfying $\tau \cdot v =0$.   
Now, by  the first few lines of the proof, 
\begin{eqnarray*}
 \langle F, v \otimes F^{T}v \rangle \tau & = & (F^{-1}v \cdot F^{T}v -F^{-1}v \otimes F^{T}v)\adj F \tau \\
& = & (F^{-1}v \cdot F^{T}v) \adj F \tau -  (F^{T}v \cdot \adj F \tau)  F^{-1} v\\
& = & FF^{-1}v \cdot v \ \adj F \tau. \\
& = & \adj F \tau,
\end{eqnarray*}
which is \eqref{jettison} in the case that $F$ is invertible.  Note that the second term of the second line of the calculation above vanishes because $F^{T}v \cdot \adj F \tau = (\det F) (F^{-T}F^{T}v \cdot \tau) =0$.   Finally, since the invertible matrices are dense in $\mathbb{R}^{3 \times 3}$, a simple approximation argument can be used to show that $\eqref{ramson}$ holds for all $F$, as claimed. 
\end{proof}

\begin{prop}\label{salix2}Let $w$ be an admissible map satisfying condition $(I)$, and let $\varphi$ be any smooth test function.  
Let $K$ be defined by 
\begin{equation}K(w)=\int_{B} \frac{x}{|x|^{2q -1}} \cdot G(w) \,dx. 
\end{equation}
Then the first variation of $K$ at $w$ in the direction $\varphi$ is given by 
\begin{equation}\label{EL1} \delta K(w)[\varphi]=\int_{B}  \frac{x}{|x|^{2q -1}} \cdot \left\{\langle \nabla w, \nabla \varphi \rangle \frac{w}{|w|^{3}}+ \frac{\adj \nabla w}{|w|^{3}}\left(\1 - 3 \bar{w} \otimes \bar{w}\right) \varphi \right\}\,dx.\end{equation} 
\end{prop}
\begin{proof}For brevity, define $\zeta(x)=\frac{x}{|x|^{2q -1}}$.
Let the functional on the right-hand side of \eqref{EL1} be $L(w)[\varphi]$.   By definition,
\[ \delta K(w)[\varphi] = \partial_{\eps}\arrowvert_{\eps=0}\int_{B}  \zeta \cdot G(w+\eps \varphi) \,dx,\]
provided the right-hand side exists.   Now
\begin{eqnarray*} \frac{1}{\eps}(K(w+\eps \varphi) - K(w)) - L(w)[\varphi] & = & A_{1} + A_{2},\end{eqnarray*}
where
\begin{align}
\label{snake} A_{1}:&=&\int_{B} \zeta \cdot \left( \frac{1}{\eps}(\adj (\nabla w + \eps \nabla \varphi) - \adj \nabla w) - \langle \nabla \varphi , \nabla w \rangle \right)\,dx, \\ 
\label{bark} A_{2}:& =&\int_{B} \zeta \cdot (\adj \nabla w )\left(\frac{1}{\eps}(h(\eps)-h(0))-(\1 - 3 \bar{w} \otimes \bar{w})\varphi\right) \,dx,\end{align}   
and
\[ h(\eps): = \frac{w+\eps \varphi}{|w+\eps \varphi|^{3}}\]
whenever $|w+\eps \varphi| \neq 0$.  Note that $h(0)$ is well defined because $|w|$ is in particular bounded away from zero on $B$.   By continuity, we may further assume that for sufficiently small $\eps$ the function $|w+\eps \varphi| \geq c$ on $B$, and hence that $h(\eps)$ is well-defined on $B$.

To estimate the integral $A_{1}$, use Definition \ref{defo1} to write 
\[ \frac{1}{\eps}\left(\adj (\nabla w + \eps \nabla \varphi) - \adj \nabla w \right)- \langle \nabla \varphi , \nabla w \rangle = \eps \ \adj \nabla \varphi. \]
Thus
\[ |A_{1}| \leq C \int_{B} \eps |\zeta||\adj \nabla \varphi|\,dx\]
for some constant $C$ depending on $c$.   Hence $A_{1} \to 0$ as $\eps \to 0$.

To estimate the integral $A_{2}$ note that
\[h'(\eps)=\frac{1}{|w+\eps \varphi|^{3}}\left(\1-3\overline{(w+\eps \varphi)} \otimes \overline{(w+\eps \varphi)}\right)\varphi,\]
so that $|h'(\eps)| \leq C |\varphi|$.
It follows by the mean value theorem that  
\[\left|\frac{1}{\eps}\left(h(\eps)-h(0)\right)\right| \leq C |\varphi|\]
for all sufficiently small $\eps$ and almost every $x$ in $B$.   Thus
\[\frac{1}{\eps}(h(\eps)-h(0))-(\1 - 3 \bar{w} \otimes \bar{w})\varphi\]
is essentially bounded on $B$.   The function $\zeta \cdot \adj \nabla w$ lies in $L^{1}(B)$ using H\"{o}lder's inequality as in Proposition \ref{salix1}.  Thus, by dominated convergence, $\lim_{\eps \to 0}A_{2}=0$, and so \eqref{EL1} follows. \end{proof} 

Our primary aim is to show that $\delta K(w)[\varphi] =0$ for as large a class of $w$ and $\varphi$ as possible. 
The proof of this hinges on an integration by parts which, since the functions involved may be of low regularity, must be handled carefully; this is the point of Lemma \ref{salix4} below, which is itself preceded by a technical result, Lemma \ref{salix3}, that  will be needed during the proof Lemma \ref{salix4}.   The functions $w^{(j)}$ appearing in Lemmas \ref{salix3} and \ref{salix4} are specially mollified versions of $w$;  their main features are summarised in the first step of the proof of Lemma \ref{salix4},  while full details 
are contained in Appendix \ref{appA}.

\begin{lem}\label{salix3} Let the function $w$ and the sequence $\{w^{(j)}\}_{j \in \mathbb{N}}$ belong to $W^{1,2q}(B;\mathbb{R}^{3})$, where $2< 2q < 3$.  For each $\eps>0$ define the integer $j(\eps)$ by \
\[j(\eps) = \left\lfloor \frac{|\ln \eps|}{\ln 2} \right\rfloor\]
and assume that, for some fixed $x_{0}$ in $B$ and all $\eps$ sufficiently small,
\begin{equation}\label{e:shosta}\int_{B(x_{0},2^{-j(\eps)})\setminus B(x_{0},2^{-(j(\eps)+1)})} \left|\nabla w^{(j(\eps))}\right|\,dx \leq \int_{B(x_{0},2^{1-j(\eps)})\setminus B(x_{0},2^{-(j(\eps)+2)})} |\nabla w|\,dx.\end{equation}
Then, for all sufficiently small $\delta >0$,
\[\essinf_{\eps \in (0,\delta)} \ \eps^{2-2q}\int_{\partial B(x_{0},\eps)} \left|\nabla w^{(j(\eps))}\right| \, d \sch^{2}  =0.\]
\end{lem}
\begin{proof} Suppose for a contradiction that the result is false.  Then there is a positive constant $M$, say, and $\delta_{0} >0$ such that
\begin{equation}\label{dvorak1}  M \leq (\eps')^{2-2q}\int_{\partial B(x_{0},\eps')} |\nabla w^{(j(\eps'))}| \, d \sch^{2} \ \ \ \quad \textrm{for a.e. }\ \eps' \in  (0,\delta_{0}). \end{equation} 
Let $\eps > 0$ and note that if $\eps'$ belongs to the interval $\omega_{\eps}:=(2^{1-j(\eps)},2^{-j(\eps)}]$ then $j(\eps')=j(\eps)$; therefore the function $j(\cdot)$ is piecewise constant and is, in particular, independent of $\eps$ on the interval $\omega_{\eps}$.  
Integrating \eqref{dvorak1} with respect to $\eps'$ over $\omega_{\eps}$ gives, with $j=j(\eps)$,
\begin{equation*}\begin{split}\frac{M}{2q-1}&(2^{(2q-1)}-1)(2^{-(j+1)})^{(2q-1)}  \leq   \int_{B(x_{0},2^{-j})\setminus B(x_{0},2^{-(j+1)})} \left|\nabla w^{(j(\eps))}\right| \,dx  \\
&   \quad \quad \quad \quad \leq C\left(\int_{B(x_{0},2^{1-j})\setminus B(x_{0},2^{-(j+2)})}|\nabla w|^{2q}\,dx\right)^{\frac{1}{2q}} (2^{-3j})^{\frac{2q-1}{2q}},\end{split}\end{equation*}
where we have used H\"{o}lder's inequality in the last line above, together with the hypothesis \eqref{e:shosta}.  The term involving $|\nabla w|$ is bounded above by $||\nabla w||_{L^{2q}(B)}$; simplifying the other terms gives
\begin{equation*} M \leq  C (2^{-j})^{\left(\frac{3}{2q}-1\right)(2q-1)}\end{equation*}
for a positive constant $C$ independent of $j(\eps)$. 
The right-hand side of this inequality is proportional to $2^{-j(\eps) \theta}$, where $\theta = (3/2q - 1)(2q-1)$ is positive and where $j(\eps) \to \infty$ as $\eps \to 0$:  it can thus be made arbitrarily small by letting $\eps \to  0$, which is a contradiction.
\end{proof}

Note that the hypotheses of the next result restrict the class of admissible maps $w$ to which we can safely apply integration by parts.

\begin{lem}\label{salix4}Let $w \in W^{1,2q}(B;\mathbb{R}^{3})$ and let $\zeta(x)=\frac{x|x|}{|x|^{2q}}$ for non-zero $x$ in $B$.   Assume that $2 < 2q <3$ and that $w$ satisfies the following strengthened version of condition (I):
\begin{itemize}\item[($I^{\prime}$)] \ \ $\exists \ x_{0} \in B$ and $\delta, \tau_{1}, \tau_{2}>0$ such that
\begin{itemize}\item[($I^{\prime}a$)]$w\arrowvert_{B\setminus \{x_{0}\}}$ is continuous; 
\item[($I^{\prime}b$)] $w(B'(x_{0},\delta)) \subset B(0,\tau_{2})\setminus B(0,\tau_{1}).$
\end{itemize}
\end{itemize} 
Then, for all $\varphi \in C_{c}^{1}(B,\mathbb{R}^{3})$,
\begin{equation}\label{bruchner7}
\int_{B} \zeta \cdot \langle \nabla w, \nabla  \varphi \rangle \frac{w}{|w|^{3}} \,dx =  \int_{B} \frac{\zeta}{|w|^{3}} \cdot \left\{ 2 (\adj \nabla w)\,\varphi - 3 \langle \bar{w} \otimes \nabla w^{T}\bar{w}, \nabla w\rangle \varphi\right\}\,dx. \end{equation}  
\end{lem}
\begin{rem}\emph{The spirit of condition (I') is that $w$ has just one discontinuity at $x_{0}$ which maps a punctured ball about $x_{0}$ to an annular region about $0$ in the target domain, and is otherwise continuous.}\end{rem} 
 

\begin{rem}\emph{The result continues to hold, and its proof requires only minor adjustments, if condition (I') is weakened to allow finitely many discontinuities.}\end{rem}

 \begin{proof} The proof of the lemma is in two steps.
 
\vspace{1mm}
 
\noindent \textbf{Step 1:}  (Mollification and continuity with respect to strong $W^{1,2q}$ convergence.)    We mollify $w$ using the technique set out in Appendix A.  The result is a sequence of functions $w^{(j)}$, say, with the properties that:
\begin{itemize}\item[(i)] for each $\eps>0$ and all sufficiently large $j$, $w^{(j)}$ is smooth on $B\setminus B'(x_{0},\eps/2)$;
\item[(ii)] each $w^{(j)}$ satisfies $|w^{(j)}| \geq \frac{\tau_{0}}{2}$ almost everywhere on $B$, and
\item[(iii)] the sequence $w^{(j)} \to w$ in $W^{1,2q}(B;\mathbb{R}^{3})$ as $j \to \infty$.
\end{itemize}
We claim that each side of \eqref{bruchner7} is continuous with respect to strong convergence in $W^{1,2q}$.   The left-hand side is the functional 
\[ F_{1}(w) :=  \int_{B} \zeta \cdot \langle \nabla w, \nabla  \varphi \rangle \frac{w}{|w|^{3}} \,dx \]
which, owing to its linear dependence on $\nabla w$, is easier to handle than both quadratic terms on the right.   For later use we define these quadratic terms by 
\[ F_{2}(w) := \int_{B} \frac{\zeta}{|w|^{3}} \cdot (\adj \nabla w)\,\varphi \,dx,\]
and
\[F_{3}(w):= \int_{B} \langle \nabla w^{T}\bar{w} \otimes \bar{w}, \nabla w\rangle \varphi \,dx.\]

\vspace{2mm}
\noindent
\textbf{Continuity of $F_{1}$:}  Consider
\begin{eqnarray*} \left|\zeta \cdot \langle \nabla w^{(j)}, \nabla  \varphi \rangle \frac{w^{(j)}}{|w^{(j)}|^{3}} - \zeta \cdot \langle \nabla w, \nabla  \varphi \rangle \frac{w}{|w|^{3}}\right| & \leq & C|\zeta|\left|\langle \nabla w^{(j)}, \nabla  \varphi \rangle \right||w-w^{(j)}| \\
& + & C|\zeta|\left|\langle \nabla w^{(j)} - \nabla w , \nabla \varphi\rangle \right|,
\end{eqnarray*}
where $C$ is a positive constant which depends on the radii $\tau_{1}>\tau_{0}$ and on $\varphi$ but not on $j$.  To control the cubic terms we have used the fact that, since both $w$ and $w^{(j)}$ are bounded away from zero (by condition (I')), 
\begin{equation}\label{sorbus} \left|\frac{w^{(j)}}{|w^{(j)}|^{3}}-\frac{w}{|w|^{3}}\right| \leq |w^{(j)}-w|(|w|^{3}+C|w|)\end{equation}
for some constant $C$ depending only on the radius $\tau_{0}$.  Then notice that condition (I') further implies that $w$ (and hence $w^{(j)}$) is essentially bounded on $B$, so that the stated upper bound holds.   It follows from a version of H\"{o}lder's inequality that
\[ \int_{B} |\zeta|\left|\langle \nabla w^{(j)}, \nabla  \varphi \rangle \right||w-w^{(j)}| \,dx \leq C ||\zeta||_{r}||\nabla w||_{2q}||w-w^{(j)}||_{(2q)^{\ast}},\]
where $r=3q/(4q-3)$ is such that $||\zeta||_{r} < \infty$.   Sobolev's embedding theorem (see, e.g., \cite[Theorem 7.10]{GT})) implies that $||w-w^{(j)}||_{(2q)^{\ast}} \to 0$.   The second term in the estimate of $F_{1}$ can be dealt with using the fact that $\nabla w^{(j)} \to \nabla w$ in $W^{1,2q}$ norm.   Hence $F_{1}(w^{(j)}) \to F_{1}(w)$ as $j \to \infty$.

\vspace{2mm}
\noindent
\textbf{Continuity of $F_{2}$:}  

Let $\eta > 0$ be arbitrary and consider the inequality
\begin{eqnarray*} \left|\zeta \cdot \frac{(\adj \nabla w^{(j)}) \varphi}{|w^{(j)}|^{3}} - \zeta \cdot \frac{(\adj \nabla w) \ \varphi}{|w|^{3}} \right| & \leq & ||w^{(j)}|^{-3}-|w|^{-3}|\left|\zeta \cdot (\adj \nabla w^{(j)}) \varphi\right| \\
& + & \left|\frac{\zeta}{|w|^{3}} \cdot \left(\adj \nabla w^{(j)} - \adj \nabla w \right) \varphi\right|,\end{eqnarray*}
which applies to the integrand of $F_{2}(w^{(j)})-F_{2}(w)$.
The first term on the right-hand side can be dealt with as follows.   Let $\eps > 0$ and note that by elementary estimates

{\small{\[\int_{B(x_{0},\eps)}||w^{(j)}|^{-3}-|w|^{-3}|\left|\zeta \cdot (\adj \nabla w^{(j)}) \varphi\right|\,dx \leq C ||\zeta||_{L^{q}(B(x_{0},\eps))}(||\nabla w^{(j)}||_{L^{2q}(B)})^{2}.\]}}
Since $\nabla w^{(j)} \to \nabla w$ in $L^{2q}$, $||\nabla w^{(j)}||$ is bounded above independently of $j$, and, 
since $\zeta \in L^{q'}(B)$, $||\zeta||_{L^{q'}(B(x_{0},\eps))} < \eta$ provided $\eps$ is sufficiently small.   Hence 
\[\int_{B(x_{0},\eps)}||w^{(j)}|^{-3}-|w|^{-3}|\left|\zeta \cdot (\adj \nabla w^{(j)}) \varphi\right|\,dx \leq C \eta\] 
 for sufficiently small $\eps$.  

It remains to estimate the integral over $B \setminus B(x_{0},\eps)$.   We use the continuity of $w$ away from $x_{0}$ to control the difference $||w^{(j)}|^{-3}-|w|^{-3}|$.  In fact, by the preceding argument and inequality \eqref{sorbus}
\[\int_{B\setminus B(x_{0},\eps)}||w^{(j)}|^{-3}-|w|^{-3}|\left|\zeta \cdot (\adj \nabla w^{(j)}) \varphi\right|\,dx \leq C \max_{B\setminus B(x_{0},\eps)}(|w^{(j)}|-|w|)\] 
for some constant $C$ depending on $w$ and $q$ but not on $j$ or $\eps$.  By the construction of $w^{(j)}$ given in Lemma \ref{a:appendixa}, we may assume $\max_{B\setminus B(x_{0},\eps)}(|w^{(j)}|-|w|) < \eta$ if $\eps$ is small enough.  Combining the estimates over $B(x_{0},\eps)$ and its complement in $B$ finally yields
\[\int_{B}||w^{(j)}|^{-3}-|w|^{-3}|\left|\zeta \cdot (\adj \nabla w^{(j)}) \varphi\right|\,dx < C\eta \ \textrm{for all sufficiently small} \ \eps.\]


The second term in \eqref{sorbus} can be estimated by writing
\[\adj \nabla w^{(j)} - \adj \nabla w = \frac{1}{2}\langle \nabla w^{(j)} - \nabla w, \nabla w^{(j)} + \nabla w \rangle,\]
so that
{\small{\[ \int_{B}\left|\frac{\zeta}{|w|^{3}}\cdot (\adj \nabla w^{(j)} - \adj \nabla w)\varphi\right|\,dx \leq C||\zeta||_{q'}||\nabla w^{(j)} + \nabla w||_{2q}||\nabla w^{(j)} - \nabla w||_{2q}.\]}} 
The right-hand side of this inequality tends to zero by the strong convergence of $\nabla w^{(j)}$ to $\nabla w$ in $L^{2q}$.  In summary, $F_{2}(w^{(j)}) \to F_{2}(w)$ as $j \to \infty$.   A similar argument works for $F_{3}$. 

\vspace{1mm}
\noindent\textbf{Step 2:} (Proof of \eqref{bruchner7} for `smooth w'.)
Let $\eps > 0$ and take $j$ so large that $w^{(j)}$ is smooth on $B\setminus B(x_{0},\eps/2)$, as in the proof of Lemma \ref{salix3}.   Consider
{\small{\begin{equation}\label{e:bruchner8}\begin{split}\int_{B\setminus B(x_{0},\eps)} \zeta \cdot \langle \nabla \wl
, \nabla  \varphi \rangle & \frac{\wl}{|\wl|^{3}} \,dx  = 
  -\int_{\partial B(x_{0},\eps)}\zeta_{i}\veps^{sab}\veps^{icd}\wl_{a,c}|\wl|^{-3}\wl_{s}\varphi_{b}\bar{x}_{d} \,d\sch^{2} \\ 
& \quad -  \int_{B\setminus B(x_{0},\eps)}\zeta_{i,d}\veps^{sab}\veps^{icd}\wl_{a,c}\wl_{s}|{\wl}|^{-3}\varphi_{b}\,dx \\
& \quad -  \int_{B\setminus B(x_{0},\eps)}\zeta_{i}\veps^{sab}\veps^{icd}\wl_{a,cd}\wl_{s}|\wl|^{-3}\varphi_{b} \,dx \\
& \quad - \int_{B\setminus B(x_{0},\eps)}\zeta_{i}\veps^{sab}\veps^{icd}\wl_{a,c}\left(\wl_{s}|{\wl}|^{-3}\right)_{,d}\varphi_{b}\,dx \\& \quad \quad \quad \quad  =:  T_{1}+T_{2}+T_{3}+T_{4}\end{split}\end{equation}}}
\noindent 
Let us deal with each term of the right-hand side of \eqref{e:bruchner8} in turn.   $T_{1}$ can be estimated by
\[ \left|\int_{\partial B(x_{0},\eps)}\zeta_{i}\veps^{sab}\veps^{icd}\wl_{a,c}|\wl|^{-3}|\wl_{s}|\varphi_{b}\bar{x}_{d} \,d\sch^{2}\right| \leq C \eps^{2-2q}\int_{\partial B(x_{0},\eps)} |\nabla \wl| \,d\sch^{2}.\]
By Lemma \ref{salix3}, there is a sequence $\eps_{m} \to 0$ such that the right-hand side of this estimate converges to zero as $m \to \infty$.  Hence $T_{1}=0$. The second term, $T_{2}$, is zero because $(\nabla \zeta)_{id}$ is symmetric while $\eps^{icd}$ is antisymmetric under the transposition $i \mapsto d$, so that $\zeta_{i,d}\eps^{icd}$ is zero almost everywhere.  Similarly, $\eps^{icd}w_{a,cd}$ is zero a.e., rendering $T_{3}=0$.   $T_{4}$ can be calculated (using, among other facts, $\eps^{sab}=-\eps^{bas}$) as follows:
\begin{equation*}\begin{split}T_{4} &= \int_{B\setminus B(x_{0},\eps)}\zeta_{i}\veps^{bas}\veps^{icd}\wl_{a,c}\wl_{s,d}
|\wl|^{-3}\varphi_{b} \,dx \\
& - 3 \int_{B\setminus B(x_{0},\eps)} \zeta_{i}\veps^{baj}\veps^{icd}\wl_{a,c}\left(\bar{w}\otimes \nabla w^{T}\bar{w}\right)_{jd}|\wl|^{-3}\varphi_{b}\,dx \\ & =
\int_{B\setminus B(x_{0},\eps)}  \zeta |\wl|^{-3} \cdot \left(2 \left(\adj \nabla \wl\right)  -3 \langle \nabla \wl, \bar{w}\otimes \nabla w^{T}\bar{w} \rangle \right)\varphi \,dx
\end{split}\end{equation*}
Letting $\eps \to 0$ , and applying the continuity of $F_{1},F_{2},$ and $F_{3}$ with respect to strong convergence in $W^{1,2q}$, we conclude that \eqref{bruchner7} holds. 
\end{proof}

\begin{prop}\label{salix5} Let $K$ be the functional defined in \eqref{K} and suppose $w$ satisfies condition $(I')$ for some $x_{0} \in B$.   Then
\[\delta K(w)[\varphi]=0.\] 
\end{prop}
\begin{proof}By Proposition \ref{salix2},
\[\delta K(w)[\varphi]= \int_{B}  \zeta \cdot \left\{\langle \nabla w, \nabla \varphi \rangle \frac{w}{|w|^{3}}+ \frac{\adj \nabla w}{|w|^{3}}\left(\1 - 3 \bar{w} \otimes \bar{w}\right) \varphi \right\}\,dx.\]
Using Lemma \ref{salix4}, we rewrite this as
\begin{equation*}\begin{split}
\delta K(w)[\varphi]&=\int_{B} \frac{\zeta}{|w|^{3}} \cdot \left\{2\, \adj \nabla w \, \varphi - 3 \langle \nabla w, \bar{w} \otimes \nabla w^{T} \bar{w}\rangle \right.\\
&\quad \quad \quad \quad \quad + \left.\adj \nabla w \, \varphi - 3 \, \adj \nabla w \ \bar{w} \otimes \bar{w} \ \varphi \right\}\,dx\\
&= 3 \int_{B} \frac{\zeta}{|w|^{3}} \cdot \left\{\adj \nabla w  \, (\1 - \bar{w} \otimes \bar{w}) - \langle \nabla w, \bar{w} \otimes \nabla w^{T} \bar{w}\rangle \right\}\varphi\,dx
\end{split}\end{equation*}
Finally, by applying Lemma \ref{l:identity} with $F:=\nabla w$ and $v:=\bar{w}$, we see that the integrand of the last line above vanishes almost everywhere.  Hence $\delta K(w)[\varphi]=0$.
\end{proof}

The preceding calculations applied only to those $w$ with a single discontinuity in $B$, i.e., to $w$ satisfying condition (I').  In order to extend the preceding analysis to maps $w$ which vanish somewhere, we now introduce a modified version, $\wc$, of $w$:

\begin{defn}\label{d:wcirc} Let $w$ be an admissible map satisfying condition (II) of Definition \ref{defadmiss}.   Let $x_{0}$ be such that $w(x_{0})=0$ and let $\tau > 0$.     Let the sets $V_{\tau}$ be given by
\[ V_{\tau} = w^{-1}(B(0,\tau)).\]
Define the family of maps $\wc(\cdot;\tau)$ by
\begin{displaymath} \wc(x;\tau) = \left\{\begin{array}{ll} \tau \frac{w(x)}{|w(x)|} & \textrm{if} \ \ x \in V_{\tau} \setminus\{x_{0}\} \\
w(x) & \textrm{if} \  \ x \in B\setminus \overline{V_{\tau}}.\end{array}\right.
\end{displaymath}
\end{defn}

\begin{rem}\label{r:smalltau}\emph{We remark that the sets $V_{\tau}$ are diffeomorphic to balls for small $\tau$.   This is because, by $(II)$,  $w:U \mapsto w(U)$ is a diffeomorphism for some neighbourhood $U$ of $x_{0}$ in $B$.  Thus if we take $\tau$ small enough that $B(0,\tau) \subset w(U)$, it follows that $V_{\tau}$ is the image of a ball under the diffeomorphism $w^{-1}$.}\end{rem}  

\noindent Thus $\overset{\circ}{w}$ opens a `cavity' about zero in the target domain. The next result shows that the functionals $K$ and $I$ do not distinguish between $w$ and $\wc$.

\begin{lem}\label{l:circapprox} Suppose that $w$ satisfies condition $(II)$ and let $\wc(\cdot,\tau)$ be as in Definition \ref{d:wcirc} above, with $\tau$ small enough for the argument of Remark \ref{r:smalltau} to apply.   
Then 
\begin{equation}\label{e:mrincredible}
G(\wc(\cdot;\tau)) = G(w) \ \ \textrm{a.e. in} \ B. \end{equation}
In particular, \[ K(\wc(\cdot,\tau)) = K(w)\]
and \[I(\wc)=I(w).\]
\end{lem}

\begin{proof}We write $\wc$ for $\wc(x,\tau)$ for brevity.  Its gradient $\nabla \wc$ satisfies
\[ \nabla \wc = \frac{\tau}{|w|}\left(\nabla w - \bar{w}\otimes \nabla w^{T} \bar{w}\right) \]
on $V_{\tau}\setminus \{x_{0}\}$.   Hence,
\begin{equation*}\begin{split}G(\wc)& = \adj \nabla \wc \frac{\wc}{|\wc|^{3}} \\ & = \frac{\tau^{2}}{|w|^{2}}\bigg(\adj \nabla w - \langle \nabla w, \bar{w}\otimes \nabla w^{T}\bar{w}\rangle\bigg)\frac{\tau w}{\tau^{3}|w|} 
\\& = \adj \nabla w \frac{w}{|w|^{3}}
\\ & = G(w). \end{split}\end{equation*}
In order to pass from the second to the third line in the above we note that 
\begin{equation}\label{e:zzero} \langle \nabla w, \bar{w}\otimes \nabla w^{T}\bar{w}\rangle w = 0.\end{equation}
To see this using Lemma \ref{l:identity}, let $F:=\nabla w$ and $v=\bar{w}$, 
so that
\[ \langle \nabla w, \bar{w}\otimes \nabla w^{T}\bar{w}\rangle= \adj \nabla w -\adj \nabla w  \ \bar{w}\otimes \bar{w};\]
applying both sides to $\bar{w}$ gives
\[  \langle \nabla w, \bar{w}\otimes \nabla w^{T}\bar{w}\rangle \bar{w} = 0,\]
 which is equivalent to \eqref{e:zzero}.   The last two lines of the lemma are immediate from the relation $G(\wc)=G(w)$.
\end{proof}

Now, we can apply Proposition \ref{salix5} to calculate $\delta K(\wc)[\varphi]$ provided $\wc$ satisfies Condition (I).   This requirement is easily met by strengthening Condition (II) to include those maps $w$ which, in addition, are continuous on $B\setminus\{x_{0}\}$; the definition is as follows: 

\begin{defn}  Condition $(II^{\prime})$:\ \ \label{d:condtwoprime}   We say that $w$ satisfies condition $(II')$ if $w$ is an admissible deformation satisfying 
condition $(II)$ of Definition \ref{defadmiss} and if, in addition, $w$ is continuous in $B \setminus \{x_{0}\}$.   Here, $x_{0}$ refers to the unique point in $B$ satisfying $w(x_{0})=0$.\end{defn}

We remark that $\wc$ is not a deformation because it fails to satisfy $\det \nabla \wc > 0$ a.e. in $B$, as can be seen either by a direct calculation or by noting that $\wc$ maps a small, punctured neighbourhood of $x_{0}$ in $B$ to the surface of a ball in $\mathbb{R}^{3}$, whence $\det \nabla \wc = 0$ near $x_{0}$ by the area formula.  Fortunately, the lack of injectivity of $\wc$ seems not to matter when calculating $\delta K(\wc)[\varphi]$; nor does it matter in the following result, where, without loss, we relax the definition 
 of admissible maps to include those which are not necessarily deformations.

\begin{thm}\label{t:main1} Let $w_{1}$ and $w_{2}$ be admissible maps satisfying either condition $(I^{\prime})$ with a single discontinuity at $x_{0}$ in $B$, or condition $(II^{\prime})$ with zero at $x_{0}$ in $B$.   Suppose that $w_{1}$ and $w_{2}$ are Lipschitz path connected, that is, 
there is a one-parameter family $(\gamma(\cdot\,;t))_{t \in [0,1]}$ of admissible maps with the properties that: 
\begin{itemize}
 \item[(i)]\begin{displaymath} \gamma(\cdot\,;0)=\left\{\begin{array}{ll} \wcone(\cdot\,;\tau_{1}) & \textrm{for some} \ \tau_{1}>0, \ \textrm{if} \  w_{1} \ \textrm{satisfies (II')} \\ w_{1}(\cdot) & \textrm{if} \  w_{1} \ \textrm{satisfies (I')}\end{array}\right.\end{displaymath}
                                                                                                                                                                                                                                      and 
\begin{displaymath} \gamma(\cdot\,;1)=\left\{\begin{array}{ll} \wctwo(\cdot\,;\tau_{2}) & \textrm{for some} \ \tau_{2}>0, \ \textrm{if} \  w_{2} \ \textrm{satisfies (II')} \\ w_{2}(\cdot) & \textrm{if} \  w_{2} \ \textrm{satisfies (I')}\end{array}\right.\end{displaymath}
\item[(ii)] each $\gamma(\cdot\;t)$ satisfies condition $(I^{\prime})$;
\item[(iii)] each $\dot{\gamma}(\cdot\,;t)$ is Lipschitz and of compact support in $B$.
\end{itemize}

\noindent Then 
\begin{equation}\label{e:pc1} K(w_{1})=K(w_{2}).\end{equation}

\noindent In particular, if $w$ is an admissible map satisfying condition $(I')$ or $(II')$ at the origin and is Lipschitz path connected to $\mathbf{i}$ then  
\begin{equation}\label{e:pc2}  I(w) \geq I(\bf{i}).\end{equation}
\end{thm}

\begin{proof}  First note that, by a simple approximation argument, the result of Lemma \ref{salix4} continues to hold true for compactly supported Lipschitz $\varphi$.    Hence, by Proposition \ref{salix5}, we find that $\delta K(w)[\varphi]=0$ for the same class of $\varphi$.   From this and hypothesis (ii) above we deduce that $\delta K(\gamma(\cdot;t))[\dot{\gamma}(\cdot;t)]=0$ for $0<t<1$, and hence that $K(\gamma(t))$ is constant on $(0,1)$.  The assumptions in hypothesis (i) together with a simple continuity argument imply that $K(\gamma(\cdot;t)) \to K(\wcone(\cdot;\tau_{1}))$ as $t \to 0+$ and $K(\gamma(\cdot;t)) \to K(\wctwo(\cdot;\tau_{2}))$ as $t \to 1-$.  Hence $K(\wcone(\cdot;\tau_{1}))=K(\wctwo(\cdot;\tau_{2}))$.  Lastly, by Lemma \ref{l:circapprox}, $K(\wcone(\cdot;\tau_{1}))=K(w_{1})$ and $K(\wctwo(\cdot;\tau_{2}))=K(w_{2})$, so that $K(w_{1})=K(w_{2})$.

Inequality \eqref{e:pc2} is \eqref{stabnec} in the case that $w$ is Lipschitz path connected to the identity $\mathbf{i}$.   To prove it, recall from \eqref{morus} that 
\[ I(w) \geq I(\id) + q (K(w)-K(\id)).\]
According to \eqref{e:pc1}, the term with prefactor $q$ vanishes.   Hence $I(w) \geq I(\mathbf{i})$.
\end{proof}

\subsection{Path-connecting maps to the identity}\label{pctoid}

In this section we give concrete examples of maps to which ideas similar to those appearing in Theorem \ref{t:main1} can be applied.    
Let $f: B \to B$ be a diffeomorphism with $f(0)=0$ and such that  $f(x)=x$ if $|x|=1$.   Fix $\rho \in (0,1)$ and define the map $w$ by 
\begin{equation}\label{earedwill}w(x) = \begin{cases}
f\left(\frac{\rho x}{|x|}\right) & \textrm{if} \ 0 < |x| \leq \rho\\
f(x) & \textrm{if} \ \rho \leq |x| \leq  1. \end{cases}
\end{equation}
It is straightforward to check that $w$ satisfies Condition (I').   Rather than trying to find a path connecting $w$ directly to $\mathbf{i}$, to which one might apply Theorem \ref{t:main1} for example, it is better to exploit the invariance of $K$ with respect to the `circle map' introduced in Definition \ref{d:wcirc}.  Recall that in terms of the parameter $\tau > 0$ the map $\wc$ is given  by
\begin{eqnarray*}\wc(x) = \left\{\begin{array}{ll}
\tau \frac{w(x)}{|w(x)|} & \textrm{if} \ 0 < |w(x)| \leq \tau\\
w(x) & \textrm{if} \ \tau \leq |w(x)| \leq  1. \end{array}\right.
\end{eqnarray*}
Thus we shall connect (by a suitably regular path) the functions 
$\wc$ and $\wid$ for fixed $\tau$.  To do so we define for each  $t \in  [0,1]$ the map $\gamma(\cdot; t): B \to B$ as follows:
\begin{equation}\label{greywill}\gamma(x,t) = \begin{cases} \tau \frac{w\left(\frac{x}{|x|^{t}}\right)}{\left|w\left(\frac{x}{|x|^{t}}\right)\right|} & \textrm{if} \ \  |x|^{t}\left|w\left(\frac{x}{|x|^{t}}\right)\right| \leq \tau \\
|x|^{t}w\left(\frac{x}{|x|^{t}}\right) & \textrm{if} \ \  |x|^{t}\left|w\left(\frac{x}{|x|^{t}}\right)\right| \geq \tau.
\end{cases}
\end{equation}
That is, $\gamma$ is the circle map of the function $x \mapsto |x|^{t}w(x |x|^{-t})$, and it satisfies 
\begin{displaymath} \left.\begin{array}{l}
\gamma(x;0)=\wc(x) \\
\gamma(x;1)=\wid(x) \end{array}\right\} 
\quad \forall x \in B\setminus\{0\}
\end{displaymath} 
Moreover, the boundary condition $w(x,t)=x$ for all $x \in  \partial B$ is obeyed for all $t$ in  $[0,1]$.  The next proposition records some further properties of $\gamma$, $w$ and  $\dot{\gamma}:=\partial_{t}\gamma$.  

\begin{prop}\label{trove}  Let $\gamma(\cdot;t): B \to B$ be defined by \eqref{greywill} and let
\[ U(\tau;t)= \left\{ x \in B: \ \ |x|^{t}\left|w\left(\frac{x}{|x|^{t}}\right)\right| < \tau \right\}.\]
Let $\alpha(x;t) = x |x|^{-t}$ for nonzero $x$ in $B$, $0 \leq t \leq 1$, and recall the notation $\bar{x}=\frac{x}{|x|}$ for nonzero $x$ in  $\mathbb{R}^{3}$.  Then:
\begin{displaymath} \nabla w(x) = \left\{\begin{array}{ll} \frac{\rho}{|x|} \nabla f(\rho \bar{x}) (\1 - \bar{x} \otimes \bar{x}) & \textrm{if} \ x \in B(0,\rho), \ x \neq 0 \\
\nabla f(x) & \textrm{if} \ x \in B \setminus B(0,\rho), 
\end{array}\right.\end{displaymath}

\begin{displaymath}
\dot{\gamma}(x;t)= \left\{
\begin{array}{ll}
\tau \ln |x| \left(\bar{w(\alpha)} \otimes \bar{w(\alpha)} - \1\right)\frac{\nabla w(\alpha) \alpha}{|w(\alpha)|} & \textrm{if} \ x \in U(\tau;t) \\
|x|^{t} \ln |x| \left(w(\alpha) - \nabla w(\alpha) \alpha\right) & \textrm{if} \ x \in  B \setminus U(\tau;t),
\end{array}\right.
\end{displaymath}
and
\begin{displaymath}
\nabla \gamma (x;t) = \left\{\begin{array}{ll} \frac{\tau}{w(\alpha)} \left(\1 - \bar{w(\alpha)} \otimes \bar{w(\alpha)}\right)\nabla w (\alpha) \nabla \alpha & \textrm{if} \ \in U(\tau;t)\\
t|x|^{t-1}w(\alpha) \otimes \bar{x} + \nabla w (\alpha) \left(1-t \bar{x} \otimes \bar{x}\right)
  & \textrm{if} \ x \in  B \setminus U(\tau;t)  \end{array}\right.
  \end{displaymath}
hold.  In particular, there is a constant $C >0$ depending only on  $\tau$ and $\rho$ such that
\begin{equation}\label{ram1} |\dot{\gamma}(x,t)| \leq C \ln |x| \quad \quad \forall x \in B, \ x \neq 0, \ t \in [0,1] \end{equation}
and
\begin{equation}\label{ram2}
|\nabla \gamma(x;t)| \leq  \begin{cases}\frac{C}{|x|} & \textrm{if} \ x \in  U(\tau;t), x \neq 0 \\
\frac{C}{|x|^{1-t}}  & \textrm{if} \ x \notin  U(\tau;t).
\end{cases}
\end{equation}

\end{prop}
\begin{proof}The formulae for the derivatives follow easily from a direct calculation whose details we omit.  The estimates \eqref{ram1} and \eqref{ram2} follow from the expressions above and by noting that 
$|\nabla w(\alpha)| \leq C|\alpha|^{-1}$ and
$|\nabla \alpha| \leq C|x|^{-t}$ for constants $C$ depending only on $\tau$ and $\rho$.  The fact that $|w(\alpha)| \geq c$ for some fixed constant $c$ is also used extensively. 
\end{proof} 
 
We aim to show that $K(\gamma(\cdot;t))$ is constant as a function of $t$.   Then the fact that $\gamma(\cdot;0) = \wc$ and $\gamma(\wid)$, together with Lemma \ref{l:circapprox}, implies that $K(w) = K(\id)$.  Finally, the idea of Theorem \ref{t:main1},  and particularly inequality \eqref{morus},  can be applied to  deduce that $I(w) \geq I(\id)$.

\begin{lem}\label{ram3} With $w$ and $\gamma$ as defined in \eqref{earedwill} and \eqref{greywill} respectively, $K(\gamma(\cdot;t))$ is constant as a function of $t$.  In particular, $I(w) \geq I(\id)$.
\end{lem}
\begin{proof} We calculate $\delta K(\gamma)[\dot{\gamma}]$ by first noting that Proposition \ref{salix2} continues to hold with $\gamma$ and $\dot{\gamma}$ in place of $w$ and $\varphi$.   The only difference is that the estimate of integral $A_{2}$ defined in \eqref{bark} requires a more careful application of the dominated convergence theorem.    Specifically, 
\[\lim_{\eps \to 0} \int_{B} \zeta(x) \cdot \adj \nabla \gamma \left(\frac{1}{\eps}\left(\frac{\gamma+\eps \dot{\gamma}}{|\gamma + \eps \dot{\gamma}|} - \frac{\gamma}{|\gamma|^{3}}\right) - (1-3 \bar{\gamma} \otimes \bar{\gamma})\dot{\gamma}\right)\,dx =  0\]
continues to hold because the integrand is dominated by $C |x|^{2-2q}|\nabla \gamma|^{2}|\dot{\gamma}|$, which  by  \eqref{ram1} and  \eqref{ram2} is of order $|x|^{2-2q}|\ln |x||$ and is hence integrable on $B$.   Arguing similarly, $A_{1}$ defined in \eqref{snake} converges to zero.  Hence
\begin{equation}
\label{tibetan} \delta K(\gamma)[\dot{\gamma}] = \int_{B}\zeta(x)\cdot \left\{\langle \nabla \gamma, \nabla \dot{\gamma} \rangle \frac{\gamma}{|\gamma|^{3}}+ \frac{\adj \nabla \gamma}{|\gamma|^{3}}\left(\1 - 3 \bar{\gamma} \otimes \bar{\gamma}\right) \dot{\gamma} \right\}\,dx.
\end{equation}
The term involving $\nabla \dot{\gamma}$ in  \eqref{tibetan} can then be handled as in Step 2 of Lemma \ref{salix4}.  We write
\[\int_{B}\zeta(x)\cdot \langle \nabla \gamma, \nabla \dot{\gamma} \rangle \frac{\gamma}{|\gamma|^{3}} \,dx = \lim_{\eps \to 0} \int_{B\setminus B(0,\eps)}\zeta(x)\cdot \langle \nabla \gamma, \nabla \dot{\gamma} \rangle \frac{\gamma}{|\gamma|^{3}} \,dx\]
and then integrate by parts over the domain $B\setminus  B(0,\eps)$.  The result is, in the notation introduced in \eqref{e:bruchner8},
\[  \int_{B\setminus B(0,\eps)}\zeta(x)\cdot \langle \nabla \gamma, \nabla \dot{\gamma} \rangle \frac{\gamma}{|\gamma|^{3}} \,dx = T_{1}+T_{2}+T_{3}+T_{4}.\]  
As before, $T_{2}$ and $T_{3}$ vanish by exploiting symmetries, while $T_{1}$ and $T_{4}$ remain  as
\begin{eqnarray*}
T_{1} &  = & - \int_{\partial B(0,\eps)} \zeta_{i}\eps^{jab}\eps^{icd}\gamma_{a,c}|\gamma|^{-3}\gamma_{j}\dot{\gamma}_{b}\bar{x}_{d} \,d\mathcal{H}^{2} \\
T_{4} & = & \int_{B\setminus B(0,\eps)}  \zeta |\gamma|^{-3} \cdot \left(2 \left(\adj \nabla \gamma\right)  -3 \langle \nabla \gamma, \bar{\gamma}\otimes \nabla \gamma^{T}\bar{\gamma} \rangle \right)\dot{\gamma} \,dx.
\end{eqnarray*}
Estimates \eqref{ram1} and \eqref{ram2} imply that, for sufficiently small $\eps$,
\[ |T_{1}| \leq C \eps^{2-2q}\eps^{-1}|\ln \eps| \eps^{2}\]
for some constant $C$ independent of $\eps$, and that the integrand of $T_{4}$ belongs to $L^{1}(B)$.  Hence on letting $\eps \to 0$ we obtain
\begin{equation}\label{figaro}
\int_{B}\zeta(x)\cdot \langle \nabla \gamma, \nabla \dot{\gamma} \rangle \frac{\gamma}{|\gamma|^{3}} \,dx = 
\int_{B}  \zeta |\gamma|^{-3} \cdot \left(2 \left(\adj \nabla \gamma\right)  -3 \langle \nabla \gamma, \bar{\gamma}\otimes \nabla \gamma^{T}\bar{\gamma} \rangle \right)\dot{\gamma} \,dx.
\end{equation}
Inserting \eqref{figaro} into \eqref{tibetan} and applying the identity \eqref{l:identity} finally gives $\delta K(\gamma) [\dot{\gamma}] =  0$.   The proof can be concluded as previously indicated.
\end{proof} 


\section{Stability when $w$ vanishes somewhere}\label{S:vanish}

In this section we study the functional $K(w)$ defined in \eqref{K} in the case that $w$ satisfies condition (II').   In such circumstances the calculation of $\delta K (w)[\varphi]$ given in Proposition \ref{salix2} breaks down; indeed, its integrand is not in general a function in $L^{1}(B)$:  see Example \ref{ex:notdiff} below.   This is hardly surprising when we recall how heavily the assumption $|w| \geq c > 0$ featured in the derivation of $\delta K(w)[\varphi]$ when condition (I') applied.   When $w$ vanishes somewhere it is nevertheless possible to make progress using so-called inner variations of $w$ of the form $w^{\eps}(x)=w(x+\eps \varphi(x))$.  The derivative $\partial_{\eps} \arrowvert_{\eps=0}K(w^{\eps})$ is calculated in Lemma \ref{l:evalbt} and applied in Proposition \ref{p:furthersmaller}, where it is shown that if the zero of the  perturbation $w^{\eps}$ of $w$ moves away (or towards) the origin then $K(w^{\eps})$ decreases (increases, respectively).  For this to hold we must insist that 
$x_{0} \neq 0$,  where $x_{0}$ is the unique zero of $w$.    As the zero $x_{0}$ approaches $0$ itself, $K$ exhibits some further unusual behaviour:  see Lemma \ref{p:laster} and the  accompanying discussion for details.


The following example demonstrates that $\delta K(w)[\varphi]$ is not in general finite when $w$ satisfies condition (II') or (II).

\begin{example}\emph{\label{ex:notdiff}Fix $x_{0} \in B$ and let $\varphi$ be a smooth test function with compact support in $B$ such that $\varphi(x_{0})=e_{1}$, with $\varphi$ constant on a ball of radius $\delta$ about $x_{0}$ and where $\delta < \frac{1}{2}|x_{0}|$.  Let $w$ satisfy condition $(II)$ in such a way that \[w(x)=F(x-x_{0}) \ \ \textrm{if} \ x \in B(x_{0},\delta)\]
for some fixed and invertible $3 \times 3$ matrix $F$.  The expression for $\delta K (w)[\varphi]$ derived in Proposition \ref{salix2} contains two terms, the first of which is zero on the ball $B(x_{0},\delta)$ because $\varphi$ is constant there.  The second term is of the form
\[ \zeta \cdot \frac{\adj \nabla w}{|w|^{3}} (\1 - 3 \bar{w}\otimes\bar{w})e_{1} = \zeta \cdot  \left(\frac{\adj F}{|F(x-x_{0})|^{3}}(\1 - 3 \overline{F(x-x_{0})}\otimes\overline{F(x-x_{0})})\right)e_{1}.\]
Change variables to $y:=F(x-x_{0})$ and consider the integral 
\[\int_{B(0,\delta')} \zeta(F^{-1}y + x_{0})\cdot \frac{F^{-1}}{|y|^{3}}(e_{1}-3\bar{y}_{1}\bar{y})\,dy.\]
By taking $\delta'$ sufficiently small and applying the contintuity of $\zeta$ at $x_{0} \neq 0$, it is clear that the dominant term is given by
\[\zeta(x_{0})\cdot F^{-1}\int_{B(0,\delta')} \frac{e_{1}-3\bar{y}_{1}\bar{y}}{|y|^{3}}\,dy.\]
The function $(e_{1}-3\bar{y}_{1}\bar{y})|y|^{-3}$ is not Lebesgue integrable in any neighbourhood of $y=0$.   To see this, note that in terms of spherical polar coordinates its angular behaviour $e_{1}-3\bar{y}_{1}\bar{y}$ is separate from the radial behaviour $|y|^{-3}$.  Thus if it were integrable then Fubini's theorem would imply
\begin{equation*}\begin{split}\int_{0}^{\delta'}|y|^{-3}\bigg(\int_{\mathbb{S}^{2}} (e_{1}-3\bar{y}_{1}\bar{y})  \,& \sin \theta \,d\theta \,d \phi \bigg)\,|y|^{2}d|y| = \\ & \int_{\mathbb{S}^{2}}\left(\int_{0}^{\delta'}|y|^{-3}|y|^{2}\,d|y|\right)(e_{1}-3\bar{y}_{1}\bar{y}) \, \sin \theta \,d\theta \,d \phi. \end{split}\end{equation*}
A contradiction is reached by noting that the bracketed inner integral on the left is zero, while that on the right is clearly infinite.   
}\end{example}

We now begin the calculation of $\partial_{\eps} \arrowvert_{\eps=0}K(w^{\eps})$, giving the proof for maps satisfying either condition (I') or (II').

\begin{prop}\label{p:innervars}Let $w$ be an admissible deformation satisfying condition (I') in $B$ with discontinuity at $x_{0} \neq 0$, or condition (II') with $w(x_{0})=0$.  Let $\varphi$ be a smooth test function with support in $B$ and define for each $\eps$ the inner variation $w^{\eps}$ of $w$ by 
\[w^{\eps}(x)=w(x+\eps \varphi(x)) \ \ \ \ \ \textrm{if} \ x \in B.\]
Suppose further that $\varphi$ is zero in a ball $B(0,\delta)$ for some $\delta$ such that $0 < \delta < |x_{0}|/2$.
Then 
\begin{equation}\label{e:EME}\begin{split} \partial_{\eps}\arrowvert_{\eps=0}K(w^{\eps})& = \int_{B} \left((2q-1)\frac{(x \cdot \varphi) x}{|x|^{2q+1}}-\frac{\varphi}{|x|^{2q-1}}\right)\cdot G(w) \,dx \\
  & - \int_{B}\zeta(x) \cdot (\nabla \varphi) \, G(w) \,dx.\end{split}\end{equation}
\end{prop}
\begin{proof} The proof is fairly standard, so details are kept to a minimum. Let $y(x;\eps)=x+\eps \varphi(x)$ and write $x=\psi(y;\eps)$.  Thus $\psi$ is the inverse of the function $x \mapsto y(x;\eps)$.   Change variables in $K(w^{\eps})$, so that
\begin{equation*}\begin{split} K(w^{\eps}) & = \int_{B}\zeta(\psi(y;\eps)) \cdot \left(\1+\eps\langle \1,\nabla \varphi(\psi(y;\eps))\rangle \right)G(w(y))\,\det \nabla_{y}\psi(y;\eps)\,dy
\\ & + \eps^{2}\int_{B}\zeta(\psi(y;\eps)) \cdot \adj \nabla \varphi (\psi(y;\eps))\,G(w(y))\,\det \nabla_{y}\psi(y;\eps)\,dy. \\&=: a(\eps)+\eps^{2}b(\eps)
\end{split}\end{equation*}
For brevity, let $\psi=\psi(y;\eps)$ in the following.  The quotient $(K(w^{\eps})-K(w))/\eps$ can then be written as 
\begin{equation*}\begin{split}
\frac{K(w^{\eps})-K(w)}{\eps} & = \int_{B}\left(\frac{\zeta(\psi)-\zeta(y)}{\eps}\right)\cdot \left(\1+\eps\langle \1, \nabla \varphi(\psi)\rangle\right) G(w(y)) \det \nabla_{y}\psi \,dy \\& + \int_{B}\zeta(y)\cdot\left(\1+\eps \langle \1, \nabla \varphi(\psi)\rangle\right)\,G(w(y))\frac{(\det \nabla_{y}\psi - 1)}{\eps}\,dy
\\ & + \int_{B} \zeta(y) \cdot \langle \1, \nabla \varphi(\psi)\rangle \,G(w(y)) \,dy + \eps b(\eps).
\end{split}\end{equation*}
By inspection, $b(\eps)$ remains bounded as $\eps \to 0$, and so the last term vanishes in the limit $\eps \to 0$.    The identity 
\[ x = \psi(x+\eps \varphi(x);\eps)\]
easily implies
\[ \det \nabla_{y}\psi = \det (1+\eps \nabla \varphi(\psi))^{-1},\]
from which it follows that
\[\frac{\det \nabla_{y} \psi - 1}{\eps} = - \Div \varphi + \eps M(|\nabla \varphi|),\]
where $M(|\nabla \varphi|)$ is some uniformly bounded polynomial in $|\nabla \varphi|$.   To handle the term involving $\zeta(\psi) -\zeta(y)$ first note that $\psi(y)=y$ provided $|y|$ is sufficiently small:  this is easy to prove using the hypothesis that $\varphi$ vanishes in a small neighbourhood of the origin.   The quotient $(\zeta(\psi) -\zeta(y))/\eps$ can then be estimated using the fact that $\zeta$ is Lipschitz away from the origin, as follows:
\begin{equation*}\begin{split} \frac{|\zeta(\psi) -\zeta(y)|}{\eps}  & \leq ||\nabla \zeta||_{L^{\infty}(B \setminus B(0,\delta'))}\frac{|\psi-y|}{\eps}  \\
& \leq C||\nabla \zeta||_{L^{\infty}(B \setminus B(0,\delta'))}\left|\partial_{\eps}\arrowvert_{\eps=0}\psi\right| \\
& \leq  C||\nabla \zeta||_{L^{\infty}(B \setminus B(0,\delta'))}|\varphi|.
\end{split}\end{equation*}
(Here we have used the fact that $\partial_{\eps}\arrowvert_{\eps=0}\psi(y;\eps) = -\varphi(y)$.)

The dominated convergence theorem then applies to the various integrals.  The result is 
\begin{equation*}\begin{split}
\partial_{\eps}\arrowvert_{\eps=0}K(w^{\eps})& = \int_{B} \left((2q-1)\frac{(y \cdot \varphi) y}{|y|^{2q+1}}-\frac{\varphi}{|y|^{2q-1}}\right)\cdot G(w(y)) \,dy \\
  & + \int_{B}\zeta(y)\cdot\left(\langle \1, \nabla \varphi(y)\rangle-\Div \varphi(y) \, \1 \right)\,G(w(y))\,dy,\end{split}\end{equation*}
which, on using the easily verifiable identity 
\[ \langle \1, \nabla \varphi \rangle = (\tr \nabla \varphi) \, \1 - \nabla \varphi, \]
gives \eqref{e:EME}.
\end{proof}

\begin{lem}\label{l:schumann}  Let $w$ be an admissible deformation satisfying condition (I') in $B$, with one discontinuity at $x_{0} \neq 0$ in $B$, or condition (II') with $w(x_{0})=0$.   Let $\varphi$ be a smooth test function satisfying the hypotheses of Proposition \ref{p:innervars}.  Then there is a sequence of functions $\{w^{(j)}\}$ smooth away from $x_{0}$ which approximate $w$ in the $W^{1,2q}$ norm and are such that
\begin{equation}\label{e:preEME}\begin{split}\int_{B} \left((2q-1)\frac{(x \cdot \varphi) x}{|x|^{2q+1}}-\frac{\varphi}{|x|^{2q-1}}\right)\cdot G(w) & \,dx - \int_{B}\zeta(x) \cdot (\nabla \varphi) \, G(w) \,dx  \\ = & \lim_{\eps \to 0}\int_{\partial B(x_{0},\eps)}(\zeta \cdot \varphi)(G(w^{(j(\eps))}) \cdot \nu) \,d\sch^{2}.\end{split}\end{equation}
In particular, $K(w^{\eps})$ is differentiable in $\eps$ at $\eps = 0$ if the limit in equation \eqref{e:preEME} exists.  In this case, 
\[\partial_{\eps}\arrowvert_{\eps = 0} K(w^{\eps}) = \lim_{\eps \to 0}\int_{\partial B(x_{0},\eps)}(\zeta \cdot \varphi)(G(w^{(j(\eps))}) \cdot \nu) \,d\sch^{2}. \]
\end{lem}

\begin{proof}We give the proof in the case that $w$ satisfies condition $(I^{\prime})$.  Let $\eps >0$ and, as in the proof of Lemma \ref{salix4}, use the approximation procedure of Lemma \ref{a:appendixa} to find a function $w^{(j(\eps))}$ with the properties that:
\begin{itemize}\item[(i)] $w^{(j(\eps))}$ is smooth on $B\setminus B'(x_{0},\eps/2)$;
\item[(ii)] each $w^{(j(\eps))}$ satisfies $|w^{(j(\eps))}| \geq \frac{\tau_{0}}{2}$ almost everywhere on $B$, and
\item[(iii)] the sequence $w^{(j(\eps))} \to w$ in $W^{1,2q}(B;\mathbb{R}^{3})$ as $j(\eps) \to \infty$ (equiv. $\eps \to 0$).  
\end{itemize}
The same procedure works in the case that $w$ satisfies condition $(II^{\prime})$, with the difference that (ii) should be replaced by 
\begin{itemize}\item[(ii')] each $w^{(j)}$ vanishes once at $x_{0}$ and is locally a diffeomorphism.    \end{itemize}
This has no impact on the following proof.

Now, it is straightforward to deduce that 
\begin{equation*}\int_{B} \left((2q-1)\frac{(x \cdot \varphi) x}{|x|^{2q+1}}-\frac{\varphi}{|x|^{2q-1}}\right)\cdot G(w)  \,dx - \int_{B}\zeta(x) \cdot (\nabla \varphi) \, G(w) \,dx  \end{equation*}
is the limit as $\eps \to 0$ of 
\begin{equation*}\int_{B\setminus B_{\eps}} \left((2q-1)\frac{(x \cdot \varphi) x}{|x|^{2q+1}}-\frac{\varphi}{|x|^{2q-1}}\right)\cdot G(w^{(j(\eps))})  \,dx - \int_{B\setminus B_{\eps}}\zeta(x) \cdot (\nabla \varphi) \, G(w^{(j(\eps))}) \,dx,\end{equation*}
where $B_{\eps}:=B(x_{0},\eps)$.  For brevity, write $\tilde{w}$ for $w^{(j(\eps))}$ in the following.   Now
\begin{equation*}\begin{split} -\int_{B\setminus B_{\eps}}\zeta(x) \cdot (\nabla \varphi) \, G(\wt) \,dx & = \int_{\partial B_{\eps}}(\zeta \cdot \varphi)(G(\wt) \cdot \nu) \,d\sch^{2} \\
& + \int_{B\setminus B_{\eps}} \left(\nabla \zeta \cdot (\varphi \otimes G(\wt)) +(\zeta \cdot \varphi) \, \Div (G(\wt))\right)\,dx,
\end{split}
\end{equation*}
where $\nu(x)=\overline{x-x_{0}}$ is the outward pointing normal to $\partial B_{\eps}$.    Notice that 
\[ \nabla \zeta \cdot (\varphi \otimes G(\wt)) = \frac{\varphi}{|x|^{2q-1}}-(2q-1)\frac{(x \cdot \varphi) x}{|x|^{2q+1}},\]
and hence
\begin{equation*}\begin{split}& \int_{B\setminus B_{\eps}} \left((2q-1)\frac{(x \cdot \varphi) x}{|x|^{2q+1}}-\frac{\varphi}{|x|^{2q-1}}\right)\cdot G(\wt)  \,dx - \int_{B\setminus B_{\eps}}\zeta(x) \cdot (\nabla \varphi) \, G(\wt) \,dx \\ & \quad \quad \quad \quad \quad \quad \quad \quad  = 
\int_{\partial B_{\eps}}(\zeta \cdot \varphi)(G(\wt) \cdot \nu) \,d\sch^{2} + \int_{B\setminus B_{\eps}} (\zeta \cdot \varphi) \, \Div (G(\wt))\,dx.
\end{split}
\end{equation*} 
Next, a short calculation shows that
\[ \Div G(\wt) = \frac{1}{|\wt|^{3}}(\adj \nabla \wt)_{jk}  \left[\nabla \wt - 3 \bar{\wt} \otimes \nabla \wt^{T}\bar{\wt}\right]_{kj} + |\wt|^{-3}\wt_{k} (\cof \nabla \wt)_{kj,j}.\]
Since $\wt$ is smooth in $B\setminus B_{\eps}$, and because $\cof \nabla \wt$ is a null Lagrangian, the second term vanishes identically.   To simplify the first we apply the identity
\[\cof F \cdot  v \otimes F^{T}v = \det F  \ \ \ \ v \in \mathbb{S}^{2}\]
with the choice $F=\nabla \wt$ and $v=\bar{\wt}$.  This gives
\begin{equation*}\begin{split} 
(\adj \nabla \wt)_{jk}  \left[\nabla \wt - 3 \bar{\wt} \otimes \nabla \wt^{T}\bar{\wt}\right]_{kj} & = \tr (\det \nabla \wt \, \1) - 3 \det \nabla \wt,
\end{split}\end{equation*}
which vanishes identically.   In conclusion, only the boundary term remains, i.e.,
\begin{equation*}\begin{split}\int_{B\setminus B_{\eps}} \left((2q-1)\frac{(x \cdot \varphi) x}{|x|^{2q+1}}-\frac{\varphi}{|x|^{2q-1}}\right)\cdot G(\wt)  \,dx &  - \int_{B\setminus B_{\eps}}\zeta(x) \cdot (\nabla \varphi) \, G(\wt) \,dx \\  & = 
\int_{\partial B_{\eps}}(\zeta \cdot \varphi)(G(\wt) \cdot \nu) \,d\sch^{2}
\end{split}
\end{equation*} 
The conclusions of the Lemma now follow from the calculations above, together with the result of Proposition \ref{p:innervars}.
\end{proof}

Next, we show that $\partial_{\eps}\arrowvert_{\eps=0}K(w^{\eps})$ can be calculated in certain cases.

\begin{lem}\label{l:evalbt}Let $w$ be an admissible map satisfying condition $(II^{\prime})$ and vanishing at $x_{0} \neq 0$.   Let $\varphi$ be a test function which vanishes in a neighbourhood of the origin not containing $x_{0}$, and let $w^{\eps}(x)=w(x+\eps \varphi(x))$ be the corresponding inner variation of $w$.   Then $\partial_{\eps}\arrowvert_{\eps = 0}K(w^{\eps})$ exists and satisfies
\begin{equation}\label{e:bt3}\partial_{\eps}\arrowvert_{\eps = 0} K(w^{\eps}) = \det \nabla w(x_{0})(\zeta \cdot \varphi)(x_{0})\int_{\mathbb{S}^{2}}\frac{1}{|\nabla w(x_{0})y|^{3}} \,d\sch^{2}(y). \end{equation}
\end{lem}
   
\begin{proof}By Lemma \ref{l:schumann},
\begin{equation}\label{e:deriveval}\partial_{\eps}\arrowvert_{\eps = 0} K(w^{\eps}) = \lim_{\eps \to 0}\int_{\partial B(x_{0},\eps)}(\zeta \cdot \varphi)(G(w^{(j(\eps))}) \cdot \nu) \,d\sch^{2}, \end{equation}
 where, owing to the hypothesis that $w$ is a diffeomorphism in a neighbourhood of $x_{0}$,  $w^{(j(\eps))} \to w$ locally uniformly in $C^{0}$ norm.    Therefore we can replace $w^{(j(\eps))}$ with $w$ in \eqref{e:deriveval}.   Let $F_{0}=\nabla w(x_{0})$ for brevity.  In the following we take  $\nu=\overline{x-x_{0}}$ on $\partial B(x_{0},\eps)$ (in keeping with the proof of Lemma \ref{l:schumann}) and again apply the hypothesis that $w$ is a diffeomorphism in a neighbourhood of $x_{0}$, giving
 \begin{equation*}\begin{split} \int_{\partial B_{\eps}}(\zeta \cdot \varphi)(G(w) \cdot \nu) \,d\mathcal{H}^{2} & = (\zeta \cdot \varphi)(x_{0})\int_{\partial B_{\eps}} \frac{\adj F_{0} \, F_{0}(x-x_{0})}{|F_{0}(x-x_{0})|^{3}}\cdot \frac{x-x_{0}}{|x-x_{0}|} \,d\sch^{2}(x)\\ &
 \quad \quad \quad \quad \quad \quad \quad \quad \quad \quad \quad \quad\quad \quad \quad \quad \quad \quad \quad \quad + \mathcal{O}(\eps)\\ 
 & = \det F_{0}\,(\zeta \cdot \varphi)(x_{0})\int_{\mathbb{S}^{2}}\frac{1}{|F_{0}y|^{3}} \,d\sch^{2}(y)+\mathcal{O}(\eps). \end{split}
 \end{equation*}
Letting $\eps \to 0$ gives \eqref{e:bt3}.
 \end{proof}

\begin{prop}\label{p:furthersmaller} Let $w$ be an admissible map satisfying condition $(II^{\prime})$ and vanishing at $x_{0} \neq 0$.   Let $\varphi$ be a test function which vanishes in a neighbourhood of the origin not containing $x_{0}$, and suppose $\varphi(x_{0}) \neq 0$.  Let $w^{\eps}(x)=w(x+\eps \varphi(x))$ be the corresponding inner variation of $w$.  Define $x(\eps)$ by 
\begin{equation}\label{e:defxeps} x(\eps)+\eps \varphi(x(\eps))=x_{0},\end{equation}
so that, for sufficiently small $\eps$, $x(\eps)$ is the unique zero of $w^{\eps}$.
Then there is $\eps_{0}>0$ such that 
\begin{itemize}\item[(i)]$K(w^{\eps})<K(w)$ for $0<\eps<\eps_{0}$ if and only if $\dot{x}(0)\cdot x_{0}  >0$; 
\item[(ii)]$K(w^{\eps})> K(w)$ for $0<\eps<\eps_{0}$ if and only if $\dot{x}(0)\cdot x_{0} < 0$.
\end{itemize}
\end{prop}

\begin{rem}\emph{The existence of a smooth curve of points $x(\eps)$ satisfying \eqref{e:defxeps} and subject to $x(0)=x_{0}$ follows from the Implicit Function Theorem.  This is the origin of the requirement that $\eps$ be sufficiently small in the statement above.}\end{rem}

\begin{rem}\emph{Because $x_{0}=x(0)$, statement (i) says that $K(w^{\eps})$ locally decreases with increasing $\eps$ if and only if $\partial_{\eps}\arrowvert_{\eps=0} |x(\eps)|^{2} > 0$.   That is,  $K(w^{\eps})$ decreases if and only if $x(\eps)$ moves further away from the origin.  A corresponding interpretation applies to statement (ii).}\end{rem}

\begin{proof} In the notation of Proposition \ref{p:innervars} we write $x(\eps)=\psi(x_{0};\eps)$.    Then, since $\partial_{\eps}\arrowvert_{\eps=0}\psi(x;0)=-\varphi(x)$ for any $x$, it is clear that $\dot{x}(0)=-\varphi(x_{0})$.   Recall that $\zeta(x_{0})=x_{0}|x_{0}|^{1-2q}$, so that
\begin{equation}\label{e:signs} \zeta(x_{0})\cdot \varphi(x_{0}) = - |x_{0}|^{1-2q} x_{0} \cdot \dot{x}(0).\end{equation}
By \eqref{e:bt3}, and since $\det \nabla w(x_{0}) >0$,  the sign of $\partial_{\eps}\arrowvert_{\eps = 0} K(w^{\eps})$ is exactly that of $\zeta(x_{0})\cdot \varphi(x_{0})$, which, by \eqref{e:signs} is opposite that of $x_{0} \cdot \dot{x}(0)$.   Parts (i) and (ii) of the proposition follow easily.   
\end{proof}

In view of Proposition \ref{p:furthersmaller}, it is tempting to conclude that $K(\bf{i})$ is not a local minimum among admissible $w$ satisfying condition $(II')$ and vanishing at $x_{0} \neq 0$.  One might expect, for example, to be able to lower the energy of the identity map (as measured by $K$) by moving its zero slightly away from the origin.   However, this is not necessarily the case.  The reason is that the derivative of $K$ with respect to inner variations, as calculated in \eqref{e:bt3}, can assume almost any behaviour as a function of $\tau=|x_{0}|$ as  $\tau \to 0$, as we shall show below.  Thus $K(w) > K(\mathbf{i})$ for $w$ `close' to $\mathbf{i}$ is in principle consistent with statements  (i) and (ii) of Proposition \ref{p:furthersmaller}, though such an energy landscape would, admittedly, seem rather strange.

\begin{prop}\label{p:laster}  There exist admissible deformations $w(\cdot\,;\tau)$ of the identity map $\mathbf{i}$ satisfying condition $(II')$ and such that: 
\begin{itemize}
\item[(a)] $w(\tau e_{1};\tau)=0$ for all $\tau \in [0,\tau_{0})$;
\item[(b)] $||w(\cdot\,,\tau)-\mathbf{i}||_{W^{1,\infty}(B)} \to 0$ as $\tau \to 0$;
\item[(c)] $K(w(\cdot\,;\tau)) \to K(\mathbf{i})$ as $\tau \to 0$; 
\item[(d)] As a function of $\tau > 0$, the derivative $\partial_{\eps}\arrowvert_{\eps = 0} K(w^{\eps})$ can be made to exhibit any behaviour as $\tau \to 0$,  where $w^{\eps}$ is a suitably chosen inner variation of $w(\cdot\,;\tau)$ satisfying the conditions of Lemma \ref{l:evalbt}.
\end{itemize}

 \end{prop}
\begin{proof}  The existence of maps $w(\cdot\,;\tau)$ satisfying parts (a)-(c) in the statement of the proposition is assured by Lemma \ref{a:appendixb} below.  It remains to prove part (d) here.  Let $f:[0,\tau_{0}) \to \mathbb{R}$ be a function.   Let $w^{\eps}$ be an inner variation of the form
\[w^{\eps}(x;\tau)=w(x+\eps \mu(x);\tau)\]
where $\mu: B \to B$ is a smooth, compactly supported function such that $\mu$ is zero in a small neighbourhood of $0$. For each fixed $\tau$ choose $\mu(\tau e_{1})$ so that $e_{1}\cdot \mu(\tau e_{1}) = f(\tau)$.   Note then that $\zeta(\tau e_{1}) \cdot \mu(\tau e_{1})=\tau^{2-2q}f(\tau)$.
Equation \eqref{e:bt3} in Lemma \ref{l:evalbt} implies that \begin{equation*}\partial_{\eps}\arrowvert_{\eps = 0} K(w^{\eps}) = \det \nabla w(\tau e_{1};\tau)(\zeta \cdot \mu)(\tau e_{1})\int_{\mathbb{S}^{2}}\frac{1}{|\nabla w(\tau e_{1})y|^{3}} \,d\sch^{2}(y). \end{equation*}
By part (b),
\[  4 \pi - \tau C \leq \det \nabla w(\tau e_{1};\tau) \int_{\mathbb{S}^{2}}\frac{1}{|\nabla w(\tau e_{1};\tau)y|^{3}}\,d\sch^{2} \leq 4\pi + C\tau\]
for a fixed $C>0$ and all $\tau \in [0,\tau_{0})$.  Thus
\begin{equation*}\big|\partial_{\eps}\arrowvert_{\eps = 0} K(w^{\eps}) - \tau^{2-2q}f(\tau)\big| \leq C\tau \ \ \ \textrm{for} \ 0 < \tau < \tau_{0}.\end{equation*}

\end{proof}

\begin{appendix}
\section{Appendix A}\label{appA}
Here we give details of the mollification and approximation procedure used in Lemma \ref{salix4}.

\begin{lem}\label{a:appendixa}Let $w \in W^{1,2q}(B,\mathbb{R}^{3})$, where $2q<3$, satisfy condition $(I^{\prime})$ with $x_{0}$ its point of discontinuity,  as defined in Lemma \ref{salix4}.   Let $\tau_{0} >0$ be such that $|w(x)| \geq \tau_{0}$ for $x \in B$, $x \neq x_{0}$, the existence of such a $\tau_{0}$ being guaranteed by condition $(I^{\prime})$. 
Then there exists a sequence of functions $w^{(j)}$ in $W^{1,2q}(B;\mathbb{R}^{3})$ such that
\begin{itemize}

\item[(i)] $w^{(j)}$ satisfies $|w^{(j)}| \geq \frac{\tau_{0}}{2}$ on $B \setminus \{x_{0}\}$;
\item[(ii)] $w^{(j)}$ is smooth in $B \setminus B(x_{0},2^{-(j+1)})$;
\item[(iii)] $w^{(j)} \to w$ in $W^{1,2q}(B;\mathbb{R}^{3})$ as $j \to \infty$.
\end{itemize}
Moreover, estimate \eqref{e:shosta} holds for this sequence.
\end{lem}

\begin{proof}Fix a positive integer $j$ and let $\eta_{j}:\mathbb{R}^{3} \to \mathbb{R}$ be a smooth function such that 
\begin{itemize}\item[(a)] $\eta_{j}(x) = 0$ if $x \in B(x_{0},2^{-(j+2)})$;
\item[(b)] $\eta_{j}(x)=1$ if $x \in B \setminus B(x_{0},2^{-(j+1)})$;
\item[(c)] $0 \leq \eta_{j}(x) \leq 1$, with $|\nabla \eta_{j}(x)|\leq c2^{j+2}$  for some $c>0$, and for all $x$ in $B(x_{0},2^{-(j+1)})\setminus B(x_{0},2^{-(j+2)})$,
\end{itemize} 
Thus $1-\eta_{j}$ is a cut-off function whose gradient has support in the annular region $B(x_{0},2^{-(j+1)})\setminus B(x_{0},2^{-(j+2)})$.    Let $l$ be a positive integer and let $\rho_{\scriptscriptstyle{l^{-1}}}$ be a standard mollifier,  where $\frac{1}{l}$ now plays the r\^{o}le of the (small) parameter of mollification.  Extend $w$ to agree with $\mathbf{i}$ outside the ball $B$.   Define $w_{l}=\rho_{\scriptscriptstyle{l^{-1}}} \ast w$ and let
\[ w^{l,j}(x) = (1-\eta_{j}(x))w(x) + \eta_{j}(x)w_{l}(x) \quad \quad \textrm{for} \ x \in B.\]
Writing $w^{l,j} = w + \eta_{j} (w_{l}-w)$ and observing that, since $w$ is continuous away from $x_{0}$, 
\[||w_{l}-w||_{L^{\infty}(B\setminus B(x_{0},2^{-(j+2)}))} < \tau_{0}/2 \]
for sufficiently large $l$, it follows that $|w^{l,j}| \geq \tau_{0}/2$ for all $l \geq l(j)$.  For later use, we may also assume $l(j) > 2^{j+3}$ for all $j$.  Hence part (i) of the lemma.

Next, by construction, $w^{l,j}$ agrees with $\eta_{j}(x)w_{l(j)}(x)$ in $B \setminus B(x_{0},2^{-(j+2)})$ and so is smooth there, from which part (i) of the lemma follows once we set $l=l(j)$.  Define
\[w^{(j)}:=w^{l(j),j}.\]

It is clear that $||w^{(j)}-w||_{2q} \leq ||w_{l(j)}-w||_{2q}$, which, by a standard property of mollifiers, converges to $0$ as $l(j) \to \infty$.  Moreover, since
\[ \nabla w^{(j)} = \nabla w + (w^{(j)} - w) \otimes \nabla \eta_{j},\]
we have
\begin{equation*}\begin{split} ||\nabla w^{(j)} - \nabla w||_{2q} & \leq ||w^{(j)}-w||_{\infty} ||\nabla \eta_{j}||_{2q} \\
& \leq C \left( \int_{B(x_{0},2^{-(j+1)}) \setminus B(x_{0},2^{-(j+2)})} (c 2^{j+2})^{2q} \,\,dx\right)^{\frac{1}{2q}}  
\\ & \leq C (2^{-j})^{(\frac{3}{2q}-1)} 
\end{split}\end{equation*}
which, since $2q <3$, tends to zero as $j \to \infty$.  Hence part (iii).

Finally, recall \eqref{e:shosta}: we wish to show that
\begin{equation*}\int_{B(x_{0},2^{-j})\setminus B(x_{0},2^{-(j+1)})} \left|\nabla w^{(j)}\right|\,dx \leq \int_{B(x_{0},(2^{1-j}))\setminus B(x_{0},(2^{-(j+2)}))}|\nabla w|\,dx.\end{equation*}
By construction, $w^{(j)}$ agrees with the mollified version $w_{l(j)}$ of $w$ on $B(x_{0},2^{-j})\setminus B(x_{0},2^{-(j+1)})$.   
But, since $l(j) > 2^{j+3}$ by construction, it is then easy to check that the estimate
\[ \int_{B(x_{0},2^{-j})\setminus B(x_{0},2^{-(j+1)})} \left|\nabla w^{j(\eps)}\right|\,dx \leq \int_{B(x_{0},2^{1-j})\setminus B(x_{0},2^{-(j+2)})} |\nabla w|\,dx\]
holds for all $\eps$.   (See Ziemer \cite[Theorem 1.6.1]{Zi} for the basic idea, adapted here to a finite domain:  the main point is that mollification takes place on a scale $1/l(j) < 2^{-(3+j)}$, which explains the slight enlargement of the region of integration in the last line above.)
In summary, \eqref{e:shosta} is satisfied.
\end{proof}

\section{Appendix B}
Here we give details of the perturbation of the identity $g(\cdot;\tau)$ used in the proof of Proposition \ref{p:furthersmaller}.

\begin{lem}\label{a:appendixb} There exists a family of diffeomorphisms $\{g(\cdot;\tau): \ \ \tau \in [0,\tau_{0})\}$ of the unit ball in $\mathbb{R}^{3}$ such that 
\begin{itemize}\item[(a)] $g(x;\tau) = \mathbf{i} \ \ \forall x \in \partial B$;
\item[(b)] each $g(\cdot,\tau)$ has a unique zero at $\tau e_{1}$;
\item[(c)] $K(g(\cdot;\tau)) \to K(\bf{i})$ as $\tau \to 0$.
\end{itemize}
\end{lem}
\begin{proof}  Let $\phi$ be a smooth function with compact support in $B(0,N\tau)$, where $N$ is a large integer to be chosen shortly,  and such that $\phi(y)=e_{1}$ if $y \in B(0,3\tau)$, $|\phi(y)|\leq 1$ for all $y$ in $B$, and with $||\nabla \phi||_{L^{\infty}(B)} \leq \frac{2}{N \tau}$.  Define the map $\rho$ by $\rho(y)=y+\tau \phi(y)$.   Then, for $N \geq 4$, say, and $\tau$ sufficiently small,  $\rho$ is a diffeomorphism such that $\rho=\bf{i}$ on $\partial B$, and the same is true of $\rho^{-1}$.  Let $g(x;\tau)=rho^{-1}(x;\tau)$.  Note that $\rho^{-1}(x)=0$ if and only if $x=\rho(0)=\tau e_{1}$.  We claim that $g$ is the desired perturbation of $\bf{i}$.   

Firstly, we use the parameter $N$ to approximate the gradient part of $G(g)$, as follows. Note that $\nabla_{x}g(x;\tau)=(\1+\tau \nabla \phi(g(x;\tau)))^{-1}$, so that
\begin{equation}\begin{split}\adj \,\nabla_{x}g(x;\tau) & = \frac{\1+\tau \nabla \phi(g(x;\tau))}{\det(\1+\tau \nabla \phi(g(x;\tau)))} \\
& = \1 +\tau(\nabla \phi(g) - \Div \phi (g)) + o\left(\frac{1}{N}\right).
\end{split}\end{equation}
Here we have used the relation $\tau |\nabla \phi(g)| \leq \frac{2}{N}$ repeatedly.   Therefore
\[\int_{B} \left|\zeta(x)\cdot \left(G(g)-\frac{g}{|g|^{3}}\right)\right| \,dx \leq \frac{c}{N}\int_{B}|\zeta(x)|(|g(x)|^{-2}) \,dx \]
for a constant $c$.  The integral on the right of the last inequality is finite and independent of $N$ and $\tau$.  To see this, change variables by setting $z=g(x;\tau)$ (which implies $x=\tau \phi(z) + z$) and estimate 
\begin{equation*}\begin{split} \int_{B}|\zeta(x)|(|g(x)|^{-2}) \,dx & \leq c \int_{B}|\tau \phi(z) + z|^{2-2q}|z|^{-2}\,dz \\ & \leq 
c \int_{B} |z|^{2-2q-2}|z|^{2}\,d|z|,\end{split}\end{equation*}
which, since $2q <3$, is finite.    Now
\begin{equation*}\begin{split}\int_{B} \bigg|\zeta(x) \cdot \bigg(\frac{g(x;\tau)}{|g(x;\tau)|^3}-& \frac{x}{|x|^{3}} \bigg)\bigg|\,dx  = \int_{B}|x|^{-2q}\left( \frac{x \cdot g(x;\tau)}{|g(x;\tau)|^{3}}-1\right)\,dx
\\& \ \leq \int_{B} |z+\tau \phi(z)|^{-2q}\left|\frac{|z+\tau \phi(z)|(z+\tau \phi(z))\cdot z}{|z|^{3}} -1 \right|\,dz 
\end{split}\end{equation*}
and so, in view of all of the above, a sufficient condition for \[ K(g(\cdot;\tau)) \to K(\mathbf{i}) \ \ \ \textrm{as} \ \tau \to 0\]
is 
\begin{equation}\label{e:swine}\int_{B} |z+\tau \phi(z)|^{-2q}\left|\frac{|z+\tau \phi(z)|(z+\tau \phi(z))\cdot z}{|z|^{3}} -1 \right|\,dz  \to 0 \ \ \textrm{as} \ \tau \to 0.\end{equation}

\vspace{2mm}
\noindent\textbf{Proof of \eqref{e:swine}:}  Recall that $\phi(y) = \tau e_{1}$ on $B(0,3\tau)$.  Let $\Pi^{-}=\{z \in B: \ |z+(\tau e_{1}/2)| < |z|\}$ and $\Pi^{+}=\{z \in B: \ |z+(\tau e_{1}/2)| > |z|\}$.   Write $P(z)$ for the integrand of \eqref{e:swine} for brevity.    Then 
\begin{equation*}\begin{split} \int_{B(0,3\tau) \cap \Pi^{-}} P(z) \,dz & \leq \int_{B(0,3\tau) \cap \Pi^{-}} |z+\tau e_{1}|^{2-2q} \left(\frac{\tau}{2}\right)^{-2} \, dz \\ & \quad \quad \quad \quad \quad \quad \quad \quad \quad \quad \quad \quad + \int_{B(0,3\tau) \cap \Pi^{-}}|z+\tau e_{1}|^{-2q}\,dz \\ &
 \leq \int_{B(0,3\tau)} |y|^{2-2q} \left(\frac{\tau}{2}\right)^{-2} \,dy + \int_{B(0,3\tau)}|y|^{-2q}\,dy
 \\ & \leq c \tau^{3-2q}.
\end{split}\end{equation*}
Similarly, this time without the change of variables in the last line,
\begin{equation*}\begin{split} \int_{B(0,3\tau) \cap \Pi^{+}} P(z) \,dz  & \leq  \int_{B(0,3\tau) \cap \Pi^{+}}  \frac{(\tau/2)^{-2q}(2\tau)^{2}}{|z|^{2}}\,dz + \int_{B(0,3\tau) \cap \Pi^{+}}(\tau/2)^{-2q} \,dz
 \\ & \leq c \tau^{3-2q}.
\end{split}\end{equation*}
Finally, we turn to $\int_{B \setminus B(0,3 \tau)} P(z)\,dz$.   First note that $P(z)=0$ when $z \notin \textrm{supp} \, \phi$, so that the integral is over $B(0,N\tau)$.   Next, define
\[ \chi  = 2 \frac{\tau}{|z|} \bar{z} \cdot \phi(z) + \frac{\tau^{2}}{|z|^{2}}|\phi(z)|^{2} \]
and note that $|\chi| \leq \frac{7}{9}$ whenever $|z| > 3\tau$.   Hence we may expand
\[|z+\tau \phi(z)| = |z| + |z|\sum_{n=1}^{\infty}b_{n}\chi^{n}\]
for $z \in B(0,N\tau)\setminus B(0,3\tau)$ and for an appropriate choice of coefficients $b_{n}$, all of which satisfy $|b_{n}| \leq 1$.   In particular, note that 
\begin{equation*}\begin{split}\frac{|z+ \tau \phi(z)|(z+ \tau \phi(z)) \cdot z}{|z|^{3}}  - 1 & =  \sum_{n=1}^{\infty} b_{n}\chi^{n}\frac{(z+ \tau \phi(z)) \cdot z}{|z|^{2}} \end{split}\end{equation*}
and
\begin{equation*}\begin{split}\int_{B(0,N\tau)\setminus B(0,3\tau)}\frac{|z+ \tau \phi(z)|^{-2q}|(z+ \tau \phi(z)) \cdot z}{|z|^{2}} \,dz& \leq c \int_{3\tau}^{N \tau} |z|^{2-2q}\,d|z| \\ & \leq
c (N\tau)^{3-2q}, \end{split}\end{equation*}
where we have used the elementary estimate $|z+\tau \phi(z)| \geq c |z|$ for $z \in B(0,N\tau)\setminus B(0,3\tau)$.
Hence, by applying a suitable convergence theorem, we obtain
\begin{equation*}\begin{split}\int_{B(0,N\tau)\setminus B(0,3\tau)} P(z) \,dz & \leq c (N\tau)^{3-2q}\sum_{n=1}^{\infty}\left(\frac{7}{9}\right)^{n}. \end{split}\end{equation*}
In conclusion, 
\[\int_{B}P(z)\,dz \leq c(1 + N^{3-2q})\tau^{3-2q}\]
for some constant $c$ independent of $N$ and $\tau$.   The right-hand side of the last estimate may be made arbitrarily small by choosing $\tau$ small enough (for fixed $N$, chosen earlier), which proves \eqref{e:swine}.

\end{proof}
\end{appendix}

\end{document}